\documentclass[12pt]{article}
\usepackage{verbatim}
\usepackage{amssymb,amsmath,amsfonts,amsthm,mathrsfs}
\usepackage{graphicx,graphics,psfrag,epsfig,calrsfs,cancel}

\usepackage{tikz}
\usetikzlibrary{shapes}

\newtheorem{remark}{Remark}
\newtheorem{lemma}{Lemma}
\newtheorem{proposition}{Proposition}
\newtheorem{theorem}{Theorem}
\newtheorem{corollary}{Corollary}

\newcommand{\R}{{\mathbb R}}
\newcommand{\C}{{\mathbb C}}

\newcommand{\eps}{{\varepsilon}}

\newcommand{\bL}{{\bf L}}

\newcommand{\ee}{{\rm e}}
\newcommand{\dif}{{\rm d}}

\newcommand{\Real}{\text{\sf Re}}

\newcommand{\uW}{\underline{W}}
\newcommand{\ubu}{\underline{\bf u}}
\newcommand{\ubV}{\underline{\bf V}}
\newcommand{\vit}{{\bf u}}
\newcommand{\bu}{{\bf u}}
\newcommand{\bU}{{\bf U}}
\newcommand{\bv}{{\bf v}}
\newcommand{\hbv}{\widehat{{\bf v}}}
\newcommand{\tbv}{\widetilde{{\bf v}}}
\newcommand{\br}{{\bf r}}

\newcommand{\hrb}{\overline{\widehat{r}}}
\newcommand{\hbr}{\widehat{{\bf r}}}
\newcommand{\bV}{{\bf V}}
\newcommand{\hv}{\widehat{v}}
\newcommand{\bw}{{\bf w}}
\newcommand{\hw}{\widehat{w}}
\newcommand{\hu}{\widehat{u}}
\newcommand{\hbw}{\widehat{{\bf w}}}
\newcommand{\bh}{{\bf h}}
\newcommand{\bn}{{\bf n}}
\newcommand{\bN}{{\bf N}}
\newcommand{\bM}{{\bf M}}
\newcommand{\bP}{{\bf P}}
\newcommand{\bQ}{{\bf Q}}
\newcommand{\bB}{{\bf B}}
\newcommand{\bS}{{\bf S}}
\newcommand{\bC}{{\bf C}}
\newcommand{\bnu}{{\boldsymbol \nu}}
\newcommand{\bxi}{{\boldsymbol \xi}}
\newcommand{\bfeta}{{\boldsymbol \eta}}
\newcommand{\bx}{{\bf x}}
\newcommand{\ual}{{u_\alpha}}
\newcommand{\ube}{{u_\beta}}
\newcommand{\uga}{{u_\gamma}}
\newcommand{\ualj}{{u_{\alpha,j}}}
\newcommand{\ubej}{{u_{\beta,j}}}
\newcommand{\ubel}{{u_{\beta,\ell}}}
\newcommand{\ugal}{{u_{\gamma,\ell}}}
\newcommand{\ugam}{{u_{\gamma,m}}}

\newcommand{\vbe}{{v_\beta}}
\newcommand{\vga}{{v_\gamma}}
\newcommand{\valj}{{v_{\alpha,j}}}
\newcommand{\vbej}{{v_{\beta,j}}}
\newcommand{\vbel}{{v_{\beta,\ell}}}
\newcommand{\vgal}{{v_{\gamma,\ell}}}
\newcommand{\vgam}{{v_{\gamma,m}}}

\newcommand{\caljbel}{{c_{\alpha j \beta \ell}}}
\newcommand{\ealbejgal}{{e_{\alpha \beta j \gamma \ell}}}
\newcommand{\ealgalbej}{{e_{\alpha \gamma \ell \beta j}}}
\newcommand{\ebealjgal}{{e_{\beta \alpha j \gamma \ell}}}
\newcommand{\egaaljbel}{{e_{\gamma \alpha j \beta \ell}}}
\newcommand{\ealjbelga}{{e_{\alpha  j \beta \ell \gamma}}}
\newcommand{\ealjbegal}{{e_{\alpha j \beta \gamma  \ell}}}
\newcommand{\daljbelgam}{{d_{\alpha j \beta \ell \gamma  m}}}

\newcommand{\Lag}{{\mathcal L}}
\newcommand{\En}{{\mathcal W}}

\newcommand{\Pun}{\mbox{\rm (C1)}}
\newcommand{\Pde}{\mbox{\rm (C2)}}

\newcommand{\Hunter}{\mbox{\rm (H)}}
\newcommand{\Hun}{\mbox{\rm (H1)}}
\newcommand{\Hde}{\mbox{\rm (H2)}}
\newcommand{\Htr}{\mbox{\rm (H3)}}
\newcommand{\Hq}{\mbox{\rm (H4)}}
\newcommand{\NLBVP}{\mbox{\rm (NLBVP)}}
\newcommand{\LBVP}{\mbox{\rm (LBVP)}}
\newcommand{\one}{\mbox{\rm (P1)}}
\newcommand{\hone}{\widehat{\mbox{\rm (P1)}}}
\newcommand{\honeadj}{{\mbox{\rm (Q1)}}}
\newcommand{\tone}{\widetilde{\mbox{\rm (P1)}}}
\newcommand{\two}{\mbox{\rm (P2)}}
\newcommand{\htwo}{\widehat{\mbox{\rm (P2)}}}

\newcommand{\tr}{\mbox{\sf tr}}

\newcommand{\sgn}{\mbox{\sf sgn}}

\usepackage{hyperref}

\begin{document}

\tikzstyle{every picture}+=[remember picture]

\title{Amplitude equations  for weakly nonlinear surface waves in variational problems}

\author{Sylvie {\sc Benzoni-Gavage}$^\dag$ \& Jean-Fran\c{c}ois {\sc Coulombel}$^\ddag$\\
$ $\\
{\small $\dag$ Universit\'e de Lyon, Universit\'e Claude Bernard Lyon 1,}\\
{\small CNRS, UMR5208, Institut Camille Jordan, 43 boulevard du 11 novembre 1918}\\
{\small F-69622 Villeurbanne-Cedex, France}\\
{\small $\ddag$ CNRS, Universit\'e de Nantes, Laboratoire de Math\'ematiques Jean Leray (CNRS UMR6629)}\\
{\small 2 rue de la Houssini\`ere, BP 92208, 44322 Nantes Cedex 3, France}\\
{\small Emails: {\tt benzoni@math.univ-lyon1.fr, jean-francois.coulombel@univ-nantes.fr}}}
\date{\today}
\maketitle

\begin{abstract}
Among hyperbolic Initial Boundary Value Problems (IBVP), those coming from a variational principle 
`generically' admit linear surface waves, as was shown by Serre [J. Funct. Anal. 2006]. At the weakly 
nonlinear level, the behavior of surface waves is expected to be governed by an amplitude equation 
that can be derived by means of a formal asymptotic expansion. Amplitude equations for weakly 
nonlinear surface waves were introduced by Lardner [Int. J. Engng Sci. 1983], Parker and co-workers 
[J. Elasticity 1985] in the framework of elasticity, and by Hunter [Contemp. Math. 1989]  for abstract 
hyperbolic problems. They consist of nonlocal evolution equations involving a complicated, bilinear 
Fourier multiplier in the direction of propagation along the boundary. It was shown by the authors in 
an earlier work [Arch. Ration. Mech. Anal. 2012] that this multiplier, or kernel, inherits some algebraic 
properties from the original IBVP. These properties are crucial for the (local) well-posedness of the 
amplitude equation, as shown together with Tzvetkov [Adv. Math., 2011]. Properties of amplitude 
equations are revisited here in a somehow simpler way, for surface waves in a variational setting. 
Applications include various physical models, from elasticity of course to the director-field system 
for liquid crystals introduced by Saxton [Contemp. Math. 1989] and studied by Austria and Hunter 
[Commun. Inf. Syst. 2013]. Similar properties are eventually shown for the amplitude equation 
associated with surface waves at reversible phase boundaries in compressible fluids, thus completing 
a work initiated by Benzoni-Gavage and Rosini [Comput. Math. Appl. 2009].
\end{abstract}

\noindent {\small {\bf AMS subject classification:} 35L53, 35L50, 74B20, 35L20.}

\noindent {\small {\bf Keywords:} surface wave,  weakly nonlinear expansion, amplitude equation, 
non-local Burgers equation, non-local Hamilton--Jacobi equation, Hamiltonian structure, 
Oseen--Frank energy, phase boundaries.}

\tableofcontents

\section{Introduction}
\label{intro}
In view of its topic and bibliography, this paper may look as though it were written in the honor of either 
John Hunter or Denis Serre. In fact, it is dedicated to a mathematician of the same generation, on the 
occasion of his 65th birthday, and this is not by chance. Guy M\'etivier has indeed been very influential 
in the work of both authors since the 1990s, and especially regarding two underlying topics in this paper, 
namely the stability of shocks and geometric optics.

Everything began with the discovery of \emph{surface waves}\footnote{Emphasized words are explained 
in the bulk of the paper.} associated with - somehow idealized - propagating {phase boundaries} 
\cite{Benzoni1998}, which thus departed from the case of \emph{classical shocks} investigated earlier by 
Majda \cite{Majda}. Surface waves are special instances of so-called \emph{neutral modes} that cannot 
occur in connection with classical shocks, but they do occur for some \emph{undercompressive} shocks 
such as \emph{reversible phase boundaries}. This fact led to several developments that are out of 
purpose here. What we are concerned with now is to gain insight on the step beyond the local-in-time 
existence results `{\`a la Majda}' for propagating discontinuities. One way is to consider \emph{weakly 
nonlinear} asymptotics on longer time scales. Regarding surface waves associated with phase boundaries, 
this approach was started in \cite{BenzoniRosini}. Earlier studies were mostly concerning surface waves in 
elasticity \cite{Lardner,Parker,ParkerTalbot}. Research on {weakly nonlinear surface waves} in more general 
hyperbolic boundary value problems was launched by Hunter \cite{Hunter}. A general feature of weakly 
nonlinear surface waves is that they are governed by a (very) complicated, nonlocal \emph{amplitude equation}. 
More recently, the authors of the present paper investigated which properties of amplitude equations could 
be inferred from the fully nonlinear boundary value problem \cite{BenzoniCoulombel}. At about the same time, 
a then student of M\'etivier managed to rigorously justify, for dissipative boundary value problems, the asymptotic 
expansion in which the leading order term corresponds to weakly nonlinear surface waves \cite{Marcou}.

Here we focus on the properties of amplitude equations for \emph{variational problems}, first for abstract 
problems and then for phase boundaries. Roughly speaking, amplitude equations associated with surface 
waves in variational problems are found to be locally well-posed. The abstract part in \S~\ref{s:varpb} 
provides in particular a way of revisiting the case of elasticity that is much simpler than in \cite{BenzoniCoulombel} 
and also applies to more general energies, such as the Oseen--Frank energy for liquid crystals considered 
by Austria and Hunter \cite{AustriaHunter,Austria}. The more specific part \S~\ref{s:pt} closes the loop about 
phase boundaries, which do not fit the abstract framework of \S~\ref{s:varpb} and may nevertheless be 
viewed as a variational problem.

\section{Amplitude equations in abstract variational problems}\label{s:varpb}
\subsection{General framework}
This paper is concerned with non-stationary models arising from a variational principle.  The most basic ones are associated with space-time Lagrangians of the form
$$\Lag[\bu]:=\int_{0}^{T}\int_{\Omega} \left(\frac12 |\bu_t|^2-W(\bu,\nabla\bu)\right)\,\dif \bx\,\dif t\,,$$
where $\Omega$ is a smooth, multidimensional domain, $\bu$ is a vector valued unknown, $\bu_t$ denotes its partial derivative with respect to $t$, and $\nabla\bu$ denotes its spatial gradient. To be more specific about notations, if $\bu(\bx,t)\in \R^n$ for $(\bx,t)\in \overline{\Omega}\times [t_1,t_2]$, $\Omega\subset\R^d$, we denote by $(u_1,\ldots,u_n)$ the components of $\bu$, and 
the entries of the matrix valued function $\nabla \bu$ are denoted by
$$u_{\alpha,j}:=\partial_{x_j} u_\alpha\,,\quad \alpha\in\{1,\ldots.,n\}\,,\;j\in\{1,\ldots,d\}\,.$$
Our first assumption on the spatial energy density $W$ is that it smoothly depends on its arguments, and satisfies the identities
\begin{itemize}
\item[\Hun]$\qquad \dfrac{\partial W}{\partial \ual}(\bu,0)\,=\,0\,,\quad \forall \bu\in \R^n\,,\;\forall \alpha\in \{1,\ldots,n\}  \,,$
\item[\Hde]$\qquad \dfrac{\partial^2 W}{\partial \ual \partial \ubej}(\bu,0)\,=\,0\,,\quad \forall \bu\in \R^n\,, \;\forall \alpha, \beta\in \{1,\ldots,n\}\,,\;\forall j\in \{1,\ldots,d\} \,.$
\end{itemize}
The identities in \Hun\ and \Hde\ are satisfied in particular when $W$ depends quadratically on $\nabla\bu$. We ask \Hun\ so as to ensure that all \emph{uniform, constant} states $\ubu$ are \emph{critical points} of both the space-time Lagrangian $\Lag$ and the spatial energy $\En$ defined by
$$\En[\bu]:=\int_{\Omega} W(\bu,\nabla\bu)\,\dif \bx\,,$$
in the sense that the variational gradients of $\Lag$ and $\En$ vanish at $\ubu$.
Let us point out indeed that the variational gradient of $\Lag$ is 
$$\delta \Lag[\bu] = - \bu_{tt}-\delta \En[\bu]\,,$$
with, using Einstein's convention on summation over repeated indices,
$$(\delta \En[\bu])_{\alpha}=\,\dfrac{\partial W}{\partial \ual}(\bu,\nabla\bu)\,-\,\left(\dfrac{\partial W}{\partial \ualj}(\bu,\nabla\bu)\right)_{\!\!\!,\,j}\,,\;\forall \alpha\in \{1,\ldots,n\}\,.$$
Thanks to \Hun\ both $\delta \Lag[\bu]$ and $\delta \En[\bu]$ vanish when $\bu\equiv \ubu$ does not depend on $(\bx,t)$.
The reason for asking \Hde\ will be given afterwards.

The variational problem we are interested in concerns the more general critical points of $\Lag$ that satisfy `natural' boundary conditions associated with $\Lag$. This was precisely the kind of problem addressed by Austria \cite{Austria} in his thesis. If we consider `test functions' $\bh$ that vanish at times $t=0$ and $t=T$, but not necessarily at the boundary $\partial \Omega$ of $\Omega$, we see that 
$$\frac{\dif }{\dif \theta} \Lag[\bu+\theta\bh]_{|\theta=0}\,=\,\int_{0}^{T}\int_{\Omega} \delta \Lag[\bu]\cdot \bh\,+\,
\int_{t_1}^{t_2}\int_{\partial\Omega} \bN[\bu]\cdot \bh\,,$$
where
$$(\bN[\bu])_\alpha:=\,\nu_j\,\dfrac{\partial W}{\partial \ualj}(\bu,\nabla\bu)\,,\;\forall \alpha\in \{1,\ldots,n\}\,,$$
and $\bnu$ denotes the unit normal vector to $\partial\Omega$ that points \emph{inside}\footnote{This unusual choice is made for convenience, so as to avoid too many minus signs in calculations.} $\Omega$. Therefore, the directional derivative here above equals zero for all $\bh$ if and only if $\delta \Lag[\bu]=0$ \emph{and} $\bN[\bu]=0$.
This is the motivation for considering the nonlinear boundary value problem
$$ \NLBVP\qquad \left\{\begin{array}{ll}\bu_{tt}+\delta \En[\bu]=0& \mbox{in } \Omega\,,\\
\bN[\bu]=0 & \mbox{on } \partial\Omega\,.\end{array}\right.$$
One may notice that the addition of a \emph{null Lagrangian}, that is, a functional of identically zero variational derivative to $\En$ leaves invariant the interior equations in \NLBVP\ but changes the boundary conditions.
This is what happens for instance with the Oseen--Frank energy
$$W(\bu,\nabla\bu)=\frac12 \alpha (\nabla\cdot\bu)^2+ \frac12 \beta (\bu\cdot (\nabla \times \bu))^2 + \frac12 \gamma |\bu\times (\nabla\times \bu)|^2+\frac12 \eta (\tr(\nabla \bu)^2 - (\nabla\cdot \bu)^2)\,,$$
in which the last term corresponds to a null Lagrangian. Up to the addition of a Lagrange multiplier associated with the constraint $|\bu|=1$ to this energy, \NLBVP\ then corresponds to a model introduced by Saxton \cite{Saxton} and Al\`{\i} and Hunter \cite{AliHunter-crystals} for nematic liquid crystals. This specific boundary value problem and a simplified version of it were studied by Austria \cite{Austria,AustriaHunter}. Otherwise, a most famous model that fits the abstract setting in \NLBVP\ is given by the equations describing hyper-elastic materials with traction free boundary condition, on which there is abundant literature.
The main purpose of this work is to shed light on the weakly nonlinear surface waves associated with \NLBVP, under minimal assumptions on the energy $W$. By staying at an abstract level we can indeed avoid many technical details, and find out which properties of the weakly nonlinear surface wave equations are inherited from the fully nonlinear boundary value problem. This was already our point of view in our earlier paper \cite{BenzoniCoulombel}. Even though variational problems may be viewed as special cases of the Hamiltonian problems considered in \cite[\S2]{BenzoniCoulombel}, the present study is at the same time simpler and more general in terms of the assumptions on the energy $W$ - for instance the Oseen--Frank energy satisfies \Hun\ and \Hde\ but not the more stringent assumptions made in \cite{BenzoniCoulombel}.

As already observed, \Hun\ ensures that uniform constant states $\ubu$ automatically satisfy the interior equations in \NLBVP. This is also true for the boundary conditions when $W$ depends quadratically on $\nabla\bu$, but for more general energies $W$  we can have $\bN[\ubu]\neq 0$. 

\subsection{Linear surface waves}
From now on, we \emph{assume} that $\ubu$ is such that $\bN[\ubu]= 0$, so that $\ubu$ solves \NLBVP.
Then small perturbations about $\ubu$ are expected to be governed by the \emph{linearized} problem
$$ \LBVP\qquad \left\{\begin{array}{ll}\bv_{tt}+ \bP \bv = 0& \mbox{in } \Omega\,,\\
\bB \bv=0 & \mbox{on } \partial\Omega\,,\end{array}\right.$$
where
$\bP:=\delta^2 \En[\ubu]$ and $\bB$ is the vector valued operator whose components $\bB_\alpha$ are defined by differentiating $(\bN[\bu])_\alpha$ at $\ubu$, which gives
$$\bB_\alpha \bv:=\nu_j\,v_\gamma\,\dfrac{\partial^2 W}{\partial \ualj \partial \uga}(\ubu,0)\,+\,\nu_j\,v_{\beta,\ell}\,\dfrac{\partial^2 W}{\partial \ualj \partial \ubel}(\ubu,0)\,.$$
This is where the assumption \Hde\ comes in. Indeed, we are interested in boundary value problems that are \emph{scale invariant}. More precisely, we would like \LBVP\ to be invariant with respect to any rescaling of the type
$(\bx,t,\bv)\mapsto (k\bx,kt,\bv)$, $k\neq 0$. Of course, the first requirement is that the domain $\Omega$ be scale invariant. 

From now on, $\Omega$ will implicitly be assumed to be a \emph{half-space}\footnote{The reader may think of $\Omega$ as $\{\bx\,;\;x_d>0\}$, so that $\bnu=(0,\ldots,0,1)$, but we prefer keeping the notations $\nu_j$ for the components of $\bnu$ in the calculations, for symmetry reasons.}. Regarding the interior equations in \LBVP, \Hun\ and a weakened version of \Hde\ would be sufficient to ensure scale invariance. As a matter of fact, the general expression for the differential operator $\delta^2 \En[\bu]$ is given by
$$(\delta^2 \En[\bu] \bv)_{\alpha} = \begin{array}[t]{l}\vbe\,\dfrac{\partial^2 W}{\partial \ual \partial \ube}(\bu,\nabla\bu)\,+\,
\vbej\,\dfrac{\partial^2 W}{\partial \ual \partial \ubej}(\bu,\nabla\bu) \\ [10pt]
-\,\left(
\vbe\,\dfrac{\partial^2 W}{\partial \ualj \partial \ube}(\bu,\nabla\bu)\,+\,
\vbel\,\dfrac{\partial^2 W}{\partial \ualj \partial \ubel}(\bu,\nabla\bu)\right)_{\!\!\!,\,j}
\,.\end{array}$$
For $\bu\equiv \ubu$ the zeroth order terms in $\delta^2 \En[\ubu]$ vanish because of \Hun , while the first order ones cancel out as soon as we have the symmetry
$$\dfrac{\partial^2 W}{\partial \ual \partial \ubej}(\ubu,0)\,=\,\dfrac{\partial^2 W}{\partial \ualj \partial \ube}(\ubu,0)\,,\quad \forall \alpha, \beta\in \{1,\ldots,n\}\,,\;\forall j\in \{1,\ldots,d\} \,.$$
We do need the stronger assumption that these derivatives are equal to zero for the boundary operator $\bB$ to be a homogeneous, first order operator. This is why we assume \Hde.
Introducing the convenient notations
$$\caljbel:=\dfrac{\partial^2 W}{\partial \ualj \partial \ubel}(\ubu,0) \,,$$
we see that under \Hun\ and \Hde\ the operators $\bP$ and $\bB$ reduce to
$$(\bP \bv)_{\alpha} = - \caljbel\,v_{\beta,\ell j}\,,\qquad (\bB \bv)_{\alpha} = \nu_j \caljbel\,v_{\beta,\ell}\,,$$
where 
$$v_{\beta,\ell j}:=\partial_{x_j} \partial_{x_\ell} v_\beta\,.$$
Remarkably enough, \LBVP\ then exactly falls within the framework considered by Serre in \cite{Serre-JFA}, up to introducing the reduced, quadratic energy density $\uW$ defined by
$$\uW(\nabla \bv):= \,\frac 12 \,\caljbel\, \valj\, \vbel\,,$$
and assuming that it is \emph{strictly rank-one convex}. This is our next assumption, which ensures that the Cauchy problem for the system $\bv_{tt}+ \bP \bv = 0$ in $\R^d$ is well-posed, whatever the chosen reference state $\ubu$.
\begin{itemize}
\item[\Htr]$\qquad v_\alpha\,\xi_j\,v_\beta\,\xi_\ell\dfrac{\partial^2 W}{\partial \ualj \partial \ubel}(\bu,0)\,>\,0\,,\quad \forall \bu\in \R^n\,, \;\forall \bv\in \R^{n}\backslash\{0\}\,,\;\forall \bxi \in \R^{d} \backslash\{0\}\,.$
\end{itemize}
About the Cauchy problem associated with \LBVP, one may summarize Serre's findings as follows.
\begin{theorem}[Serre \cite{Serre-JFA}]\label{thm:Serre}
Under assumptions \Hun-\Hde-\Htr, the Cauchy problem associated with \LBVP\ is always \emph{strongly well-posed} in one space dimension ($d=1$), and in arbitrary space dimensions, it is strongly well-posed in $\dot{H}^1(\Omega)$ if and only if the global energy
$$\int_{\Omega} \uW(\nabla \bv)\, \dif \bx$$
is \emph{convex} and \emph{coercive} on $\dot{H}^1(\Omega)$. If this is the case, then for all for all $\bfeta\neq 0$ in an open subset of the cotangent space to $\partial\Omega$, there exists $\tau\in \R$, $\tau\neq 0$, such that (LBVP) admits nontrivial solutions of the form 
$$\bv(\bx,t)=\,\ee^{i(\tau t + \bfeta\cdot\bx)}\,\bV(\bnu\cdot \bx)\,,\; \bV\in L^2(\R^+)\,.$$
The time frequency $\tau$ depends on the wave vector $\bfeta$ and solves the equation $\Delta(\tau,\bfeta)=0$, where 
$\Delta$ is the \emph{Lopatinskii determinant} associated with \LBVP. In addition, if the space of surface waves associated with $(\tau,\bfeta)$ is \emph{one-dimensional} then $\tau$ is a \emph{simple} root of $\Delta(\cdot,\bfeta)$, that is, $\partial_\tau\Delta(\tau,\bfeta)\neq 0$. Finally, the surface wave profile $\bV$ solves an ODE
$\bV_z=\bS(\tau,\bfeta)\bV$, where the $n\times n$ matrix $\bS(\tau,\bfeta)$ is \emph{stable}, in the sense that its eigenvalues are of negative real part.
\end{theorem}

The results stated in Theorem \ref{thm:Serre} follow from Theorems 3.1, 3.3, 3.5, and Proposition 4.1 in \cite{Serre-JFA}.
Roughly speaking, they mean that if (LBVP) does not admit any `exploding' mode solution then it admits \emph{surface waves}, which propagate with speed $\tau/|\bfeta|$ in `generic' directions $\bfeta$ along the boundary $\partial\Omega$, and decay to zero away from the boundary. They even decay exponentially fast, that is, the square integrable functions $\bV$ decay exponentially fast at infinity since they are of the form $\bV(z)=\ee^{z \bS(\tau,\bfeta)}\bV(0)$ with $\bS(\tau,\bfeta)$ a stable matrix, which amounts to the fact that the zeroes of $\Delta$ lie in the so-called \emph{elliptic} frequency domain.

\subsection{Weakly nonlinear asymptotics}

Once we have linear surface waves, it is natural to try and understand the influence of nonlinearities on their evolution. In this respect, we look for solutions of \NLBVP\ admitting a (formal) weakly nonlinear expansion 
$$\bu(\bx,t)=\ubu\,+\,\eps\,\bv(\tau t + \bfeta\cdot\bx,\bnu\cdot \bx,\eps t)\, + \,\eps^2\,\bw(\tau t + \bfeta\cdot\bx,\bnu\cdot \bx,\eps t)\,+\,{\mathcal O}(\eps^3)\,,$$
where $\tau$ and $\bfeta$ are of course related by $\Delta(\tau,\bfeta)=0$, and
$\bv=\bv(y,z,s)$ and $\bw=\bw(y,z,s)$ are supposed to be bounded as well as their derivatives in the tangential variable $y\in \R$ and the slow time $s$, and square integrable in the transverse variable $z\in \R^+$.
By plugging this expansion into \NLBVP\ we see that for all $s$ the first order profile $\bv(\cdot,\cdot,s)$ must be solution to 
$$ \one
\qquad \left\{\begin{array}{ll}\tau^2 \bv_{yy}+ \bP^\bfeta \bv = 0\,,& z>0\,,\\
\bB^\bfeta \bv=0\,, &z=0\,,\end{array}\right.$$
where the operators $\bP^\bfeta$ and $\bB^\bfeta$ are obtained from the operators $\bP$ and $\bB$ involved in \LBVP\ merely by replacing each derivative $\partial_{x_j}$ by 
$\nu_j \partial_z+\eta_j \partial_y$. Linear surface waves yield special solutions of \one\ of the form
$$\bv(y,z)=\,\ee^{i y}\,\bV(z)\,.$$ 
More generally, we can find all the solutions of \one\ by Fourier transform in $y$, under the following assumption.
\begin{itemize}
\item[\Hq]The pair $(\tau,\bfeta)\in \R^d$, with $\tau\neq 0$ and $\bfeta$ cotangent to $\partial\Omega$, is such that there are no \emph{normal mode} solutions to $\tau^2 \bv_{yy}+ \bP^\bfeta \bv = 0$ of the form $\bv(y,z)=\,\ee^{i y}\,\ee^{i \xi z} \,\ubV$ with $\xi\in \R$, $\ubV\neq 0$, and the space of solutions to \one\ of the form $\bv(y,z)=\,\ee^{i y}\,\bV(z)$ with $\bV\in L^2(\R^+)$ is one-dimensional.
\end{itemize}

In other words, \Hq\ asks that $(\tau,\bfeta)\in \R^d$ be associated with a line, and not a greater space, of surface waves. 

\begin{lemma}\label{lem:solvone}
Under assumptions \Hun-\Hde-\Hq, the space of square integrable, real-valued solutions to \one\ is made of functions of the form $w*_y\br$, where $\br$ is defined by its $y$-Fourier transform
$$\hbr(k,z) = \left\{\begin{array}{ll}\bV(kz)\,,&\;k>0\,,\\
\overline{\bV(-kz)}\,,&\;k<0\,,\end{array}\right.$$
for all $z>0$, with $\bV\in L^{2}(\R^+)$  such that $\bv(y,z)=\,\ee^{i y}\,\bV(z)$ is a fixed, nontrivial linear surface wave solution to \one.
\end{lemma}

\begin{proof}
By Fourier transform in $y$, if we denote by $k$ the dual variable to $y$, \one\ is equivalent to 
$$ \hone
\qquad \left\{\begin{array}{ll}k^2\tau^2 \hbv = \bL^{k\bfeta} \hbv\,,& z>0\,,\\
\bC^{k\bfeta} \hbv=0\,, & z=0\,,\end{array}\right.$$
where the operators $\bL^{k\bfeta}$ and $\bC^{k\bfeta}$ are obtained respectively from $\bP^{\bfeta}$ and $\bB^{\bfeta}$ by substituting $ik$ for $\partial_y$. More explicitly, they are defined by
$$(\bL^{k\bfeta} \bv)_{\alpha} = - \caljbel\,(\nu_j\partial_z+ik\eta_j)(\nu_\ell\partial_z+ik\eta_\ell) (v_{\beta})\,,\quad (\bC^{k\bfeta} \bv)_{\alpha} = \nu_j \caljbel\,(\nu_\ell\partial_z+ik\eta_\ell) (v_{\beta})\,.$$

Because of \Hun-\Hde, \LBVP\ is invariant by the rescaling 
$(\bx,t,\bv)\mapsto (k\bx,kt,\bv)$ for all $k> 0$. Since \one\ is obtained from \LBVP\ by setting 
$y=\tau t+ \bfeta\cdot \bx$, $z=\bnu\cdot\bx$,
this implies that $\bv=\bv(y,z)$ is solution to \one\ if and only if $\tbv=\bv(y,kz)$ is solution to
$$ \tone
\qquad \left\{\begin{array}{ll}k^2\tau^2 \tbv_{yy}+ \bP^{k\bfeta} \tbv = 0\,,& {z}>0\,,\\
\bB^{k\bfeta} \tbv=0\,, &{z}=0\,.\end{array}\right.$$
In particular, $\bv=\ee^{iy}\,\bV(z)$ is solution to \one\ if and only if $\tbv=\ee^{iy}\,\bV(kz)$ is solution to
$$  \left\{\begin{array}{ll}k^2\tau^2 \tbv = \bL^{k\bfeta} \tbv \,,& {z}>0\,,\\
\bC^{k\bfeta} \tbv=0\,, &{z}=0\,.\end{array}\right.$$
Substituting the notation $\tbv$ for $\hbv$, this is exactly $\hone$ at fixed $k$. The latter thus has 
a one-dimensional space of solutions, since this is the case for the solutions of the form $\bv=\ee^{iy}\,\bV(z)$ of \one, by \Hq. To make this more precise, let us denote 
by $\bv_0(y,z)=\,\ee^{i y}\,\bV_0(z)$ a nontrivial linear surface wave solution to \one, using temporarily the subscript $0$ to avoid confusion with other solutions to \one.
Then, for any solution to \one, for all $k>0$, there must exist a scalar $\widehat{w}(k)$ such that $\hbv(k,z)= \widehat{w}(k) \bV_0(kz)$. Furthermore, in order $\bv$ to be real-valued, we must have
$\hbv(k,z)=\overline{\hbv(-k,z)}$ for all $k<0$. 

To conclude, we remove the subscript $0$ from $\bV_0$, and define $\br$ as claimed. 
By complex conjugation we see that for any solution to $\hone$, $\bh=\overline{\hbv}$ satisfies
$$\honeadj \qquad \left\{\begin{array}{ll}k^2\tau^2\, \bh = \bL^{-k\bfeta}\, \bh\,,& z>0\,,\\
\bC^{-k\bfeta} \,\bh=0\,, & z=0\,.\end{array}\right.$$
In particular, this implies that $\br$ solves $\hone$ for all $k\neq0$ - and not only for $k>0$.
Then all square integrable, real-valued solutions $\bv$ to \one\ are such that $\hbv(k,z)= \widehat{w}(k) \hbr(k,z)$, for all $k\neq 0$ and all $z>0$.
We conclude by inverse Fourier transform.
\end{proof}

Note that for all $z>0$, $\hbr(k,z)$ is exponentially decaying when $k\to \infty$, since this is the case for $\bV(z)$ when $z\to +\infty$, and that $\hbr(k,z)$ is as smooth in $k$ as $\bV$ in $z$, except at $k=0$.
More importantly here, the fact that $\hbr(-k,z)=\overline{\hbr(k,z)}$ is solution to $\honeadj$
is crucial for the symmetry properties of the amplitude equation studied below.

Recalling that the first order profile $\bv$ in the asymptotic expansion of $\bu$ must solve \one\ and is allowed to depend on the slow time $s$, Lemma \ref{lem:solvone} shows that its general form is $\bv(\cdot,z,s)=w(\cdot,s)*\br(\cdot,z)$. Now, by plugging the expansion in \NLBVP\ we find that the second order profile $\bw$ must solve
$$ \two
\qquad \left\{\begin{array}{ll}\tau \bv_{ys}\,+\,\tau^2 \bw_{yy}+ \bP^\bfeta \bw + \frac12  \bQ^\bfeta[\bv] = 0\,,& z>0\,,
\\[10pt]
\bB^\bfeta \bw + \frac12 \bM^\bfeta[\bv] =0\,, &z=0\,,\end{array}\right.$$
where the quadratic operators $\bQ^\bfeta$ and $\bM^\bfeta$ are obtained by differentiating twice $\delta\En[\bu]$ and $\bN[\bu]$ respectively, which yields the operators $\bQ$ and $\bM$ detailed below,
and by replacing each derivative $\partial_{x_j}$ by 
$\nu_j \partial_z+\eta_j \partial_y$.
In order to write explicitly $\bQ$ and $\bM$ in a rather short way, let us introduce a few more notations, for the third order derivatives of $W$ that do not automatically vanish under the assumptions \Hun-\Hde,
$$\ealbejgal:=\dfrac{\partial^3 W}{\partial \ual \partial \ubej \ugal}(\ubu,0) \,,\;\daljbelgam:=\dfrac{\partial^3 W}{\partial \ualj \partial \ubel \ugam}(\ubu,0)\,.$$
Then we have, under \Hun-\Hde,
\begin{equation}\label{eq:defQM}
\begin{array}{r}
(\bQ[\bv])_\alpha\,=\,\ealbejgal\, \vbej\, \vgal\,-\,(\ebealjgal\,\vbe\,\vgal\,+\,\egaaljbel\,\vbel\,\vga\,+\,\daljbelgam\,\vbel\,\vgam)_{\!\!\!,\,j}\,,\\ [5pt]
(\bM[\bv])_\alpha\,=\,(\,\ebealjgal\,\vbe\,\vgal\,+\,\egaaljbel\,\vbel\,\vga\,+\,+\,\daljbelgam\,\vbel\,\vgam)\,\nu_j\,.
\end{array}
\end{equation}
(We could of course notice that $\ebealjgal\,\vbe\,\vgal\,+\,\egaaljbel\,\vbel\,\vga=2 \ebealjgal\,\vbe\,\vgal$, but it is more convenient, for symmetry reasons, to keep these two sums.)

\subsection{Derivation of amplitude equations}\label{ss:deramp}

\begin{theorem}
\label{mainthm} We assume that \Hun-\Hde-\Htr-\Hq\ hold true, and introduce $\hbr$ as in Lemma \ref{lem:solvone}.
For \two\ to have a square integrable solution $\bw$ the amplitude $\widehat{w}$ must solve the quadratic, nonlocal equation
$$a(k) \,\hw_s(k,s)\,+\,\int_{\R}\,b(-k,k-\xi,\xi)\, \hw(k-\xi,s)\,\hw(\xi,s)\,\dif \xi\,=\,0\,,$$
with 
$$a(k):=i\,c\,\sgn(k)\,,\;c:=\,\textstyle\tau\,\int_{0}^{+\infty} |\hbr(1,\zeta)|^2\,\dif \zeta\,,\quad \forall \;k\neq 0\,,$$
$$b(\xi_1,\xi_2,\xi_3)=b_1(\xi_1,\xi_2,\xi_3)+b_2(\xi_1,\xi_2,\xi_3)\,,$$
$$4\pi\,b_2(\xi_1,\xi_2,\xi_3):=$$
$$\textstyle\int_{0}^{+\infty} \daljbelgam\,(\nu_j\rho_{\alpha,z}+i\xi_1\eta_j\rho_\alpha)\,(\nu_\ell\rho_{\beta,z}+i\xi_2\eta_\ell\rho_\beta)\,(\nu_m\rho_{\gamma,z}+i\xi_3\eta_m\rho_\gamma)\dif z\,,$$
$$4\pi\,b_1(\xi_1,\xi_2,\xi_3):=$$
$$\textstyle\int_{0}^{+\infty} \ealbejgal (\nu_j \nu_\ell \rho_\alpha \rho_{\beta,z} \rho_{\gamma,z} +  i \xi_2 \eta_j \nu_\ell \rho_\alpha \rho_\beta \rho_{\gamma,z} +  i \xi_3 \eta_\ell \nu_j \rho_\alpha \rho_{\beta,z} \rho_\gamma -  \xi_2 \xi_3 \eta_j \eta_\ell \rho_\alpha \rho_\beta \rho_\gamma)\dif z\,+\,$$
$$\textstyle\int_{0}^{+\infty} \ebealjgal (\nu_j \nu_\ell \rho_{\alpha,z} \rho_{\beta} \rho_{\gamma,z} + i\xi_1 \eta_j \nu_\ell \rho_\alpha \rho_\beta \rho_{\gamma,z} +  i \xi_3 \eta_\ell \nu_j \rho_{\alpha,z} \rho_{\beta} \rho_\gamma - \xi_1 \xi_3 \eta_j \eta_\ell \rho_\alpha \rho_\beta \rho_\gamma)\dif z\,+\,$$
$$\textstyle\int_{0}^{+\infty} \egaaljbel (\nu_j \nu_\ell \rho_{\alpha,z} \rho_{\beta,z} \rho_{\gamma} + i\xi_1 \eta_j \nu_\ell \rho_\alpha \rho_{\beta,z} \rho_{\gamma} + i \xi_2 \eta_\ell \nu_j \rho_{\alpha,z} \rho_{\beta} \rho_\gamma - \xi_2 \xi_1 \eta_j \eta_\ell \rho_\alpha \rho_\beta \rho_\gamma)\dif z\,,$$
for $\xi_1\xi_2\xi_3\neq 0$, where we have used the shortcuts
$$\rho_\alpha:=\widehat{r}_\alpha(\xi_1,z)\,,\;\rho_\beta:=\widehat{r}_\beta(\xi_2,z)\,,\;\rho_\gamma:=\widehat{r}_\gamma(\xi_3,z)\,.$$
$$\rho_{\alpha,z}:=\partial_z\widehat{r}_\alpha(\xi_1,z)\,,\;\rho_{\beta,z}:=\partial_z\widehat{r}_\beta(\xi_2,z)\,,\;\rho_{\gamma,z}:=\partial_z\widehat{r}_\gamma(\xi_3,z)\,.$$
In particular, we have $$a(-k)=\overline{a(k)}\neq 0\,,\;\forall k\neq 0\,,$$
$$b(-\xi_1,-\xi_2,-\xi_3)=\overline{b(\xi_1,\xi_2,\xi_3)}\,,\;\forall (\xi_1,\xi_2,\xi_3)\,,\;\xi_1\xi_2\xi_3\neq 0\,,$$
and $b$ is symmetric - that is, $b(\xi_1,\xi_2,\xi_3)$ is invariant under all permutations of $\{\xi_1,\xi_2,\xi_3\}$.
Furthermore, under the additional assumption that the matrix $\bS(\tau,\bfeta)$ from Theorem \ref{thm:Serre} has no Jordan blocks, the part $b_1$ of $b$ is positively homogeneous degree one, while 
$b_2$ is positively homogeneous degree two.
\end{theorem}
\begin{proof}
By Fourier transform in $y$, \two\ is equivalent to 
$$ \htwo
\qquad \left\{\begin{array}{ll}k^2\tau^2 \hbw- \bL^{k\bfeta} \hbw\,=\,ik\tau \hbv_{s}\, + \frac12  \widehat{\bQ^\bfeta[\bv]}\,,& z>0\,,
\\[10pt]
  \bC^{k\bfeta} \hbw = - \frac12 \widehat{\bM^\bfeta[\bv]}  \,, &z=0\,,\end{array}\right.$$
For this problem to have a $z$-square integrable solution $\hbw$, the right-hand side must satisfy a Fredholm-type condition, and it turns out that this condition can be simply written in terms of $\hbr$. Indeed, an integration by parts shows the identity, for all $\bh$ and $\hbw$,
$$\textstyle\int_{0}^{+\infty} \bh \cdot \bL^{k\bfeta} \hbw\,\dif z \,-\,(\bh\cdot \bC^{k\bfeta} \hbw)_{|z=0} \,=\,\textstyle\int_{0}^{+\infty} (\bL^{-k\bfeta} \bh) \cdot  \hbw\,\dif z \,-\,( (\bC^{-k\bfeta}\bh)\cdot \hbw)_{|z=0}\,,$$
which obviously reduces to 
$$\textstyle\int_{0}^{+\infty} \big(- k^2\tau^2 \bh \cdot\hbw \,+\,\bh \cdot \bL^{k\bfeta} \hbw\big)\,\dif z \,=\,(\bh\cdot \bC^{k\bfeta} \hbw)_{|z=0}$$
if $\bh$ solves $\honeadj$. As already observed, this is the case for $\bh=\overline{\hbr}$. We thus find that for $\hbw$ to solve $\htwo$, we must have
$$ \textstyle\int_{0}^{+\infty} \overline{\hbr}\cdot ( ik\tau \hbv_{s}  + \frac12  \widehat{\bQ^\bfeta[\bv]})\,\dif z\,=\,
\frac12 (  \overline{\hbr} \cdot \widehat{\bM^\bfeta[\bv]})_{|z=0}\,.$$
The next important observation is that, since $\bQ^\bfeta$ and $\bM^\bfeta$ are closely related to each other, the right hand-side here above can be `absorbed' back into the integral.  Indeed, recall that $\bQ^\bfeta$ and $\bM^\bfeta$ are obtained from $\bQ$ and $\bM$ - defined in \eqref{eq:defQM} - by substituting $\nu_j \partial_z+\eta_j \partial_y$ for each derivative $\partial_{x_j}$, so that we can write
$$(\bQ^\bfeta[\bv])_\alpha\,=\,\Psi_\alpha\,-\,(\nu_j \partial_z+\eta_j \partial_y)(\Phi_{\alpha}^j)\,,\quad
(\bM^\bfeta[\bv])_\alpha\,=\,\nu_j\,\Phi_{\alpha}^j\,,$$
$$\Psi_\alpha:=\ealbejgal\, (\nu_j \partial_z+\eta_j \partial_y)(\vbe)\, (\nu_\ell \partial_z+\eta_\ell \partial_y)(\vga)\,,$$
$$\Phi_\alpha^j:=\begin{array}[t]{l}\ealjbegal \vbe (\nu_\ell \partial_z+\eta_\ell \partial_y)(\vga) + \ealjbelga \vga (\nu_\ell \partial_z+\eta_\ell \partial_y)(\vbe)  \\[5pt] + \daljbelgam (\nu_\ell \partial_z+\eta_\ell \partial_y)(\vbe) (\nu_m \partial_z+\eta_m \partial_y)(\vgam) .\end{array}$$
Hence by integration by parts,
$$\textstyle\int_{0}^{+\infty}  \hrb_\alpha (\widehat{\bQ^\bfeta[\bv]})_\alpha\,\dif z\,=\,\textstyle\int_{0}^{+\infty}  \hrb_\alpha (\widehat{\Psi}_\alpha\,-\,ik\eta_j \widehat{\Phi_{\alpha}^j})\dif z\,+\,
\int_{0}^{+\infty}  (\partial_z\hrb_\alpha) (\nu_j \widehat{\Phi_{\alpha}^j})\dif z\,+\,(  \hrb_\alpha (\widehat{\bM^\bfeta[\bv]})_\alpha)_{|z=0}\,.$$
Therefore, the equation that $\bv$ must satisfy reads
$$ik\tau\, \textstyle\int_{0}^{+\infty} \overline{\hbr}\cdot \hbv_{s} \dif z\, + \frac12
\int_{0}^{+\infty}  \hrb_\alpha \widehat{\Psi}_\alpha\dif z\,+ \frac12
\int_{0}^{+\infty} \,(-ik\eta_j)  \hrb_\alpha \widehat{\Phi_{\alpha}^j}\dif z\,+\,
 \frac12 \int_{0}^{+\infty}  (\partial_z\hrb_\alpha) (\nu_j \widehat{\Phi_{\alpha}^j})\dif z\,=\,0\,.
$$
Since $\bv=w*_y\br$, the first integral equivalently reads
$w_s \textstyle\int_{0}^{+\infty} |\hbr|^2 \dif z$,
and
$$\int_{0}^{+\infty} |\hbr(k,z)|^2\,\dif z\,=\,\int_{0}^{+\infty} |\hbr(1,kz)|^2\,\dif z\,=\,\frac{1}{k}\,\int_{0}^{+\infty} |\hbr(1,\zeta)|^2\,\dif \zeta\,,\forall k>0\,,$$
$$\int_{0}^{+\infty} |\hbr(k,z)|^2\,\dif z\,=\,\int_{0}^{+\infty} |\hbr(1,-kz)|^2\,\dif z\,=\,-\frac{1}{k}\,\int_{0}^{+\infty} |\hbr(1,\zeta)|^2\,\dif \zeta\,,\forall k<0\,,$$
hence the definition of $$a(k):=i\,\sgn(k)\,\tau\,\int_{0}^{+\infty} |\hbr(1,\zeta)|^2\,\dif \zeta,\;k\neq 0\,,$$
where $\sgn(k)$ denotes the sign of $k$.
Since $\Psi_\alpha$ and $\Phi_\alpha^j$ are all quadratic in $\bv$,
it just remains to read the contribution of the three other integrals to the amplitude equation by substituting $w*_y\br$ for $\bv$ and by using repeatedly the formula
$2\pi \widehat{uv}=\widehat{u} *\widehat{v}$. This yields the claimed, lengthy expression for the kernel $$b(-k,k-\xi,\xi)=b_1(-k,k-\xi,\xi)+b_2(-k,k-\xi,\xi)\,.$$ Both $b_1$ and $b_2$ turn out to be symmetric in their arguments thanks to the symmetries in the coefficients $\ealbejgal$ and $\daljbelgam$. It is indeed clear from the symmetries of $\daljbelgam$ that each term 
$$\daljbelgam\,(\nu_j\rho_{\alpha,z}+i\xi_1\eta_j\rho_\alpha)\,(\nu_\ell\rho_{\beta,z}+i\xi_2\eta_\ell\rho_\beta)\,(\nu_m\rho_{\gamma,z}+i\xi_3\eta_m\rho_\gamma)$$
in the sum involved in $b_2(\xi_1,\xi_2,\xi_3)$ is invariant under the transpositions 
$(\alpha,j,\xi_1)\leftrightarrow (\beta,\ell,\xi_2)$ and $(\beta,\ell,\xi_2)\leftrightarrow (\gamma,m,\xi_3)$.
The symmetry of $b_1(\xi_1,\xi_2,\xi_3)$ is a little bit trickier to check. In fact, we can see by recalling the meaning of the notations
$$\rho_\alpha=\widehat{r}_\alpha(\xi_1,z)\,,\;\rho_\beta=\widehat{r}_\beta(\xi_2,z)\,,\;\rho_\gamma=\widehat{r}_\gamma(\xi_3,z)\,,$$
and by using that
$$\ealbejgal=\ealgalbej\,,\;\forall \alpha,\beta,\gamma\in \{1,\ldots,n\}\,,\;\forall j,\ell\in\{1,\ldots,d\}\,,$$
that the twelve sums that are summed altogether to define $b_1(\xi_1,\xi_2,\xi_3)$ are either invariant or pairwise exchanged by the transpositions 
$(\alpha,\xi_1)\leftrightarrow (\beta,\xi_2)$ and $(\beta,\xi_2)\leftrightarrow (\gamma,\xi_3)$. This is shown on the pictures below.

\noindent
\begin{center}
\fbox{\begin{minipage}{147mm}
$\tikz[baseline]{\node[fill=gray!10,anchor=base] (n01){$\ealbejgal$};} (\tikz[baseline]{\node[fill=gray!10,anchor=base] (n1)
            {$\nu_j \nu_\ell \rho_\alpha \rho_{\beta,z} \rho_{\gamma,z}$};} 
            +  \tikz[baseline]{\node[fill=gray!10,anchor=base] (n2)
            {$i \xi_2 \eta_j \nu_\ell \rho_\alpha \rho_\beta \rho_{\gamma,z}$};} 
            +  \tikz[baseline]{\node[fill=gray!10,anchor=base] (n3)
            {$i \xi_3 \eta_\ell \nu_j \rho_\alpha \rho_{\beta,z} \rho_\gamma$};} 
            -  \tikz[baseline]{\node[fill=gray!10,anchor=base] (n4)
            {$\xi_2 \xi_3 \eta_j \eta_\ell \rho_\alpha \rho_\beta \rho_\gamma$};})\,+\,$\\
            
            \vspace{5mm}
$\tikz[baseline]{\node[fill=gray!10,anchor=base] (n02){$\ebealjgal$};} (\tikz[baseline]{\node[fill=gray!10,anchor=base] (n5)
            {$\nu_j \nu_\ell \rho_{\alpha,z} \rho_{\beta} \rho_{\gamma,z}$};} 
            + \tikz[baseline]{\node[fill=gray!10,anchor=base] (n6)
            {$ i\xi_1 \eta_j \nu_\ell \rho_\alpha \rho_\beta \rho_{\gamma,z}$};} 
            + \tikz[baseline]{\node[fill=gray!10,anchor=base] (n7)
            {$i \xi_3 \eta_\ell \nu_j \rho_{\alpha,z} \rho_{\beta} \rho_\gamma$};} 
            - \tikz[baseline]{\node[fill=gray!10,anchor=base] (n8)
            {$\xi_1 \xi_3 \eta_j \eta_\ell \rho_\alpha \rho_\beta \rho_\gamma$};})\,+\,$\\
            
            \vspace{5mm}
$\tikz[baseline]{\node[fill=gray!20,anchor=base] (n03){$\egaaljbel$};} (\tikz[baseline]{\node[] (n9)
            {$\nu_j \nu_\ell \rho_{\alpha,z} \rho_{\beta,z} \rho_{\gamma}$};} 
            + \tikz[baseline]{\node[fill=gray!20,anchor=base] (n10)
            {$ i\xi_1 \eta_j \nu_\ell \rho_\alpha \rho_{\beta,z} \rho_{\gamma}$};} 
            + \tikz[baseline]{\node[fill=gray!20,anchor=base] (n11)
            {$ i \xi_2 \eta_\ell \nu_j \rho_{\alpha,z} \rho_{\beta} \rho_\gamma$};} 
            - \tikz[baseline]{\node[] (n12)
            {$\xi_2 \xi_1 \eta_j \eta_\ell \rho_\alpha \rho_\beta \rho_\gamma$};})\,.\,\,\,\,$\\

\vspace{5mm}
\centerline{Effect of transposition  $(\alpha,\xi_1)\leftrightarrow (\beta,\xi_2)$}             
\begin{tikzpicture}[overlay]
	\path[<->] (n10) edge [in=-150,out=-30] (n11);
	\path[<->] (n01) edge [in=90,out=-90] (n02);
	\path[<->] (n1) edge [in=90,out=-90] (n5);
	\path[<->] (n2) edge [in=90,out=-90] (n6);
	\path[<->] (n3) edge [in=90,out=-90] (n7);
	\path[<->] (n4) edge [in=90,out=-90] (n8);
\end{tikzpicture}                 
\end{minipage}}
\end{center}

\vspace{5mm}

\noindent
\begin{center}
\fbox{\begin{minipage}{147mm}
$\tikz[baseline]{\node[fill=gray!20,anchor=base] (n01){$\ealbejgal$};} (\tikz[baseline]{\node[] (n1)
            {$\nu_j \nu_\ell \rho_\alpha \rho_{\beta,z} \rho_{\gamma,z}$};} 
            +  \tikz[baseline]{\node[fill=gray!20,anchor=base] (n2)
            {$i \xi_2 \eta_j \nu_\ell \rho_\alpha \rho_\beta \rho_{\gamma,z}$};} 
            +  \tikz[baseline]{\node[fill=gray!20,anchor=base] (n3)
            {$i \xi_3 \eta_\ell \nu_j \rho_\alpha \rho_{\beta,z} \rho_\gamma$};} 
            -  \tikz[baseline]{\node[] (n4)
            {$\xi_2 \xi_3 \eta_j \eta_\ell \rho_\alpha \rho_\beta \rho_\gamma$};})\,+\,$\\
            
            \vspace{5mm}
$\tikz[baseline]{\node[fill=gray!10,anchor=base] (n02){$\ebealjgal$};} (\tikz[baseline]{\node[fill=gray!10,anchor=base] (n5)
            {$\nu_j \nu_\ell \rho_{\alpha,z} \rho_{\beta} \rho_{\gamma,z}$};} 
            + \tikz[baseline]{\node[fill=gray!10,anchor=base] (n6)
            {$ i\xi_1 \eta_j \nu_\ell \rho_\alpha \rho_\beta \rho_{\gamma,z}$};} 
            + \tikz[baseline]{\node[fill=gray!10,anchor=base] (n7)
            {$i \xi_3 \eta_\ell \nu_j \rho_{\alpha,z} \rho_{\beta} \rho_\gamma$};} 
            - \tikz[baseline]{\node[fill=gray!10,anchor=base] (n8)
            {$\xi_1 \xi_3 \eta_j \eta_\ell \rho_\alpha \rho_\beta \rho_\gamma$};})\,+\,$\\
            
            \vspace{5mm}
$\tikz[baseline]{\node[fill=gray!10,anchor=base] (n03){$\egaaljbel$};} (\tikz[baseline]{\node[fill=gray!10,anchor=base] (n9)
            {$\nu_j \nu_\ell \rho_{\alpha,z} \rho_{\beta,z} \rho_{\gamma}$};} 
            + \tikz[baseline]{\node[fill=gray!10,anchor=base] (n10)
            {$ i\xi_1 \eta_j \nu_\ell \rho_\alpha \rho_{\beta,z} \rho_{\gamma}$};} 
            + \tikz[baseline]{\node[fill=gray!10,anchor=base] (n11)
            {$ i \xi_2 \eta_\ell \nu_j \rho_{\alpha,z} \rho_{\beta} \rho_\gamma$};} 
            - \tikz[baseline]{\node[fill=gray!10,anchor=base] (n12)
            {$\xi_2 \xi_1 \eta_j \eta_\ell \rho_\alpha \rho_\beta \rho_\gamma$};})\,.\,\,\,\,$\\

\vspace{5mm}
\centerline{Effect of transposition  $(\beta,\xi_2)\leftrightarrow (\gamma,\xi_3)$}             
\begin{tikzpicture}[overlay]
	\path[<->] (n2) edge [in=-150,out=-30] (n3);
	\path[<->] (n02) edge [in=90,out=-90] (n03);
	\path[<->] (n5) edge [in=90,out=-90] (n9);
	\path[<->] (n6) edge [in=90,out=-90] (n10);
	\path[<->] (n7) edge [in=90,out=-90] (n11);
	\path[<->] (n8) edge [in=90,out=-90] (n12);
\end{tikzpicture}                 
\end{minipage}}
\end{center}

\vspace{5mm}

The behavior of $a$ and $b$ regarding conjugation is a straightforward consequence of their definition and of the definition of $\hbr(-k,z)=\overline{\hbr(k,z)}$.

We can find the homogeneity properties of $b_1$ and $b_2$ by recalling that 
$$\hbr(k,z)\,=\,\hbr(1,kz)\,,\;k>0\,,$$
and by observing that under our additional assumption on the matrix $\bS(\tau,\bfeta)$, 
$$\hbr(1,kz)\,=\,\ee^{kz \bS(\tau,\bfeta)}\bV(0)$$ is a linear combination of \emph{exponential functions}\footnote{By exponential function of $z$ we mean a function of the form $\ee^{\omega z}$, with $\Real \omega<0$ here.} of $kz$.
Then we see that $b_1$ consists of the $z$-integral a sum of `homogeneous' terms, made of products of exponential functions of $m_i z$ involving either two $z$-derivatives, or only one multiplied by one of the frequencies $\xi_1,\xi_2,\xi_3$, or no $z$-derivative but the product of two frequencies, while $b_2$ consists of the $z$-integral of a sum of products of exponential functions of $m_i z$ in which the sum of the number of $z$-derivatives and the number of frequencies  equals three. 
\end{proof}

\begin{remark}{\rm 
It was pointed out in our earlier paper \cite{BenzoniCoulombel} that, at least in the frameworks considered there (namely, a Hamiltonian framework that includes the variational problems considered here when the coefficients $\ebealjgal$ are equal to zero on the one hand, and general boundary value problems associated with first order hyperbolic systems on the other hand), the coefficient $a$ in front of the time derivative in the amplitude equation is proportional to the derivative $\partial_\tau\Delta(\tau,\bfeta)$ of the Lopatinskii determinant. 
In Theorem \ref{mainthm}, we see that the coefficient $a$ is automatically nonzero. This is consistent with the fact proved by Serre (see Theorem \ref{thm:Serre} stated above) that $\partial_\tau\Delta(\tau,\bfeta)$ is nonzero as soon as the space of linear surface waves is one-dimensional, in variational frameworks. In more general frameworks, the condition $\partial_\tau\Delta(\tau,\bfeta)\neq 0$ is more stringent than the requirement of having a one-dimensional space of linear surface waves.
}
\end{remark}

\begin{corollary}
In the framework of Theorem \ref{mainthm}, the amplitude equation associated with surface waves of \NLBVP\  is endowed with a \emph{Hamiltonian structure}, and admits a conservation law associated with translation invariance in the direction of propagation.
\end{corollary}

\begin{proof} The amplitude equation derived in Theorem \ref{mainthm} equivalently reads
\begin{equation}\label{eq:amplitude}
w_s + {\mathcal H}( {\mathcal B}[w]) \,=\,0\,,
\end{equation}
where ${\mathcal H}$ denotes the Hilbert transform, defined by $\widehat{{\mathcal H}(v)}(k)=-i\sgn(k)\hv(k)$ for all $v\in L^2$,
and ${\mathcal B}$ is defined at least for Schwartz functions by
$$\widehat{{\mathcal B}[w]}(k)=\frac{1}{c}\,\int b(-k,k-m,m)\,\hw(k-m)\,\hw(m)\,\dif m\,.$$
Up to dividing $\br$ by $\sqrt{|c|}$, we can even assume that $c=1$, which we do from now on.
As was already pointed out in \cite{AliHunterParker},
because of the symmetry of the kernel $b$ the quadratic nonlocal operator ${\mathcal B}$ can be identified with the variational derivative of the functional ${\mathcal T}$ defined by
$${\mathcal T}[w]=\frac{1}{3 %c
}\,\iint b(-k-m,k,m)\,\hw(-k-m)\,\hw(k)\,\hw(m)\,\dif k\,\dif m\,.$$
For completeness, this is shown in the appendix in a more precise analytical framework, see Proposition \ref{prop:varderham}.

Therefore, \eqref{eq:amplitude} can be written as 
$$w_s + {\mathcal H}( \delta{\mathcal T}[w]) \,=\,0\,,$$
which is Hamiltonian since ${\mathcal H}$ is skew-adjoint. 

Similarly as the momentum pointed out by Hunter in a periodic setting \cite{Hunter2006}, the quantity
$${\mathcal M}[w]:=\frac12 \int |k| \,|\hw(k,s)|^2\,\dif k$$ is conserved along (smooth) solutions of \eqref{eq:amplitude}.
Indeed, if $w$ is a smooth solution of \eqref{eq:amplitude} we have
$$\partial_s\left(\frac12 \int |k| \,\hw(-k,s)\,\hw(k,s)\,\dif k\right) = i \int k \;\widehat{ {\mathcal B}[w]}(k,s)\,\hw(-k,s)\,\dif k$$
$$= i \iint k \;b(-k,k-m,m)\,\hw(-k,s)\,\hw(k-m,s)\,\hw(m,s)\,\dif m\,\dif k$$
$$=-\tfrac{i}{3}\;\iint(-k+k-m+m) \;b(-k,k-m,m)\,\hw(-k,s)\,\hw(k-m,s)\,\hw(m,s)\,\dif m\,\dif k\,=\,0\,,$$
by the symmetries of $b$.
Futhermore, ${\mathcal M}$ is associated with $y$-translation invariance in that
$$w_y=-\tfrac{1}{2\pi}\,{\mathcal H}(\delta {\mathcal M}[w])\,,$$
see Proposition \ref{prop:vardermom}.
\end{proof}

\begin{remark}{\rm 
\begin{enumerate}
\item When all the coefficients $\daljbelgam$ are equal to zero then
$b_2\equiv 0$ and $b$ is positively homogeneous degree one. This is what happens when the energy $W$ is quadratic in $\nabla \bu$, as for instance in the simplified model for liquid crystals studied by
Austria and Hunter \cite{AustriaHunter,Austria}, for which they find a kernel of the form
$$b_1(k,\ell,m)=(A- i B \sgn(k \ell m)) \frac{|k\ell|+|\ell m|+|m k|}{|k|+|\ell|+|m|}\,+\,
(C- i D \sgn(k \ell m)) \frac{k\ell+\ell m+m k}{|k|+|\ell|+|m|}\,,$$
and point out that the special case $A=-C=2$, $B=D=0$ reduces to
$$b_1(k,\ell,m)= |k|+|\ell|+|m|\,, \;\forall k,\ell,m\,,\;k+\ell+m=0\,.$$
The more complicated kernel associated with the full model for liquid crystals is also of degree one, see \cite[pp.~101-103]{Austria} for explicit formulas.
\item When the coefficients $\ealbejgal$ are equal to zero, which happens when $W$ depends only on $\nabla \bu$ and not on $\bu$, then $b$ is positively homogeneous degree two.
Weakly nonlinear Rayleigh waves in elasticity correspond to a seminal example of this situation.  A simplified version of the amplitude equation associated with elastic waves is named after Hamilton, Il'insky, and Zabolotskaya \cite{HamiltonIlinskyZabolotskaya}, and has a kernel of the form
$$b_2(k,\ell,m)=\,\frac{|k \ell m|}{|k|+|\ell|+|m|}.$$
\end{enumerate}
}
\end{remark}

\subsection{Well-posedness theory of amplitude equations}

Despite its nice algebraic form, the amplitude equation \eqref{eq:amplitude} is not easy to deal with from the analytical point of view. As far as the existence of classical solutions is concerned, it is important to derive \emph{a priori} estimates without loss of derivatives. One of them is for free, and is given by the conservation of the momentum ${\mathcal M}$. Indeed, if $w$ is a smooth solution of \eqref{eq:amplitude} then $u:=|\partial_y|^{1/2}w$, defined as the inverse Fourier transform of the mapping $k\mapsto |k|^{1/2}\widehat{w}(k,s)$ at each time $s$, is such that
the $L^2$ norm of $u$ is independent of $s$. Furthermore, as was pointed out by Hunter in a periodic setting \cite{Hunter2006}, the new unknown $u$ is to be sought as a solution of a nonlocal equation that is more amenable to a priori estimates than \eqref{eq:amplitude}. As a matter of fact, 
the amplitude equation \eqref{eq:amplitude} reads\footnote{Recall that we have set $c=1$ without loss of generality.}, in Fourier variables,
$$\hw_s(k,s)\,-\,i\,\sgn(k)\,\int_{\R}\,b(-k,k-m,m)\, \hw(k-m,s)\,\hw(m,s)\,\dif m\,=\,0\,,$$
which is equivalent to
$$\hu_s(k,s)\,-\,i\,k\,\int_{\R}\,p(-k,k-m,m)\, \hu(k-m,s)\,\hu(m,s)\,\dif m\,=\,0\,,$$
where the new kernel $p$ is defined by
$$p(k,\ell,m):=\frac{b(k,\ell,m)}{|k\ell m|^{1/2}}\,,\quad k\ell m\neq 0\,.$$
We observe that $p$ has the same symmetries as $b$. The advantage of the equation on $\hu$ is that the Fourier multiplier $k$ is easier to handle than $\sgn(k)$. The drawback is that the kernel $p$ is more singular than $b$. However, Hunter identified some conditions on $p$ ensuring a priori estimates without loss of derivatives for the associated nonlocal equation, when $p$ is positively homogeneous. We state them below in the cases we are concerned with, which correspond to the kernels $p_1$ and $p_2$ associated with $b_1$ and $b_2$.
These conditions read
\begin{itemize}
\item[\Pun]$\qquad |p_1(k,\ell,m)|\lesssim 1/ (\min (|k|,|\ell|,|m|))^{1/2}  \,,$
\item[\Pde]$\qquad |p_2(k,\ell,m)|\lesssim (\min (|k|,|\ell|,|m|))^{1/2}  \,,$
\end{itemize}
where the symbol $\lesssim$ means $\leq$ up to a multiplicative constant. We claim that \Pun\ and \Pde\ are indeed satisfied by the kernels obtained from Theorem \ref{mainthm}, and that they imply the local well-posedness of the amplitude equation \eqref{eq:amplitude} in $H^2$.

\begin{lemma}\label{lem:pp}
In the framework of Theorem \ref{mainthm}, if the matrix $\bS(\tau,\bfeta)$ from Theorem \ref{thm:Serre} has no Jordan blocks, the kernels $p_1$ and $p_2$ defined by
$$p_1(k,\ell,m):=\frac{b_1(k,\ell,m)}{|k\ell m|^{1/2}}\,,\;p_2(k,\ell,m):=\frac{b_2(k,\ell,m)}{|k\ell m|^{1/2}}\,,\;k\ell m\neq 0\,.$$
satisfy \Pun\ and \Pde\ respectively.
\end{lemma}

\begin{proof}
The kernel
$b_2(\xi_1,\xi_2,\xi_3)$ is a linear combination of terms of the form
$$\xi_1\xi_2\xi_3 \int_{0}^{+\infty}  \ee^{-(\omega_1|\xi_1| +\omega_2|\xi_2|+\omega_3|\xi_3|) z}\,\dif z\,=\,
\frac{\xi_1\xi_2\xi_3}{\omega_1|\xi_1| +\omega_2|\xi_2|+\omega_3|\xi_3|}\,,$$
where $\omega_1$, $\omega_2$, $\omega_3$ are eigenvalues of $-\bS(\tau,\bfeta)$ of positive real part. 
Therefore, $|b_2(\xi_1,\xi_2,\xi_3)|$ is bounded by a sum of terms
$$\lesssim  \frac{|\xi_1\xi_2\xi_3|}{\min(\Real \omega_1,\Real\omega_2,\Real\omega_3)\,(|\xi_1| +|\xi_2|+|\xi_3|)}\,,$$
hence
$$|p_2(\xi_1,\xi_2,\xi_3)|\lesssim  \frac{|\xi_1\xi_2\xi_3|^{1/2}}{|\xi_1| +|\xi_2|+|\xi_3|}\,,$$
and $$\frac{|\xi_1\xi_2\xi_3|^{1/2}}{|\xi_1| +|\xi_2|+|\xi_3|}\leq |\xi_1|^{1/2}\frac{|\xi_2\xi_3|^{1/2}}{|\xi_2|+|\xi_3|} \leq \tfrac12 |\xi_1|^{1/2}$$ by Young's inequality. By permuting the roles of $\xi_1$,
$\xi_2$, $\xi_3$, we get  that $p_2$ satisfies \Pde\ . The reasoning is similar for $p_1$. Indeed, the kernel $b_1$ is a linear combination of terms of the form
$$\xi_j\xi_{j+1} \int_{0}^{+\infty}  \ee^{-(\omega_1|\xi_1| +\omega_2|\xi_2|+\omega_3|\xi_3|) z}\,\dif z\,=\,
\frac{\xi_j\xi_{j+1}}{\omega_1|\xi_1| +\omega_2|\xi_2|+\omega_3|\xi_3|}\,,\;j=1,2,3\,,$$
where $\xi_4=\xi_1$ (and also $\xi_5=\xi_2$ below) for convenience,
so that 
$$|p_1(\xi_1,\xi_2,\xi_3)|\lesssim  \sum_{j=1}^{3}\frac{|\xi_j\xi_{j+1}|^{1/2}}{|\xi_{j+2}|^{1/2}(|\xi_1| +|\xi_2|+|\xi_3|)}\,.$$
Since 
$$\frac{|\xi_j\xi_{j+1}|^{1/2}}{|\xi_{j+2}|^{1/2}(|\xi_1| +|\xi_2|+|\xi_3|)}\leq \frac{1}{|\xi_{j+2}|^{1/2}}\,\frac{|\xi_j\xi_{j+1}|^{1/2}}{|\xi_j| +|\xi_{j+1}|} \leq  \frac{1}{2|\xi_{j+2}|^{1/2}}$$
for all $j\in \{1,2,3\}$, this shows that $p_1$ satisfies \Pun\ .
\end{proof}

\begin{theorem}\label{thm:wellposed}
In the framework of Theorem \ref{mainthm} (with $c=1$ without loss of generality), assuming that the matrix $\bS(\tau,\bfeta)$ from Theorem \ref{thm:Serre} has no Jordan blocks, let us consider the nonlocal equation governing $u=|\partial_y|^{1/2}w$,
\begin{equation}\label{eq:amplitudeb}
u_s = ( {\mathcal P}[u])_y\,,
\end{equation}
where ${\mathcal P}$ is defined by
$$\widehat{{\mathcal P}[u]}(k)=\int p(-k,k-m,m)\,\hu(k-m)\,\hu(m)\,\dif m\,,
\quad p(k,\ell,m):=\frac{b(k,\ell,m)}{|k\ell m|^{1/2}}\,,\;k\ell m\neq 0\,.$$
Then \eqref{eq:amplitude} is locally well-posed in $H^{\sigma}(\R)$, $\sigma>2$.
\end{theorem}

\begin{proof}
As said above, it is crucial to derive a priori estimates without loss of derivatives.
We already know from the conservation of ${\mathcal M}$ along solutions of \eqref{eq:amplitude} that for classical solutions to \eqref{eq:amplitudeb} we have
$$\frac{\dif }{\dif s} \|u\|_{L^2}^2 = 0\,,$$
and we claim that for all $\rho>0$ and $\sigma>2$,
\begin{equation}
\label{eq:estrho}
\frac{\dif }{\dif s} \|u\|_{\dot{H}^\rho}^2 \lesssim \|u\|_{H^\sigma}\,\|u\|_{\dot{H}^\rho}^2\,,
\end{equation}
for all $u=|\partial_y|^{1/2}w$ with $w$ a classical solution to \eqref{eq:amplitude}, where
$$\|u\|_{{H}^\sigma} := \left(\int_\R (1+|k|^{2})^{\sigma}\, |\hu(k)|^2\,\dif k\right)^{1/2}\,,\quad \|u\|_{\dot{H}^\rho} := \left(\int_\R |k|^{2\rho}\, |\hu(k)|^2\,\dif k\right)^{1/2}\,.$$
These estimates can be derived similarly as in \cite{Hunter2006}, by using the inequalities
\begin{equation}
\label{eq:est0}
\begin{array}[t]{r}|k|k|^{2\rho}+\ell|\ell|^{2\rho}+m|m|^{2\rho}|\lesssim |k|\,|\ell|^{\rho}\,|m|^\rho+|\ell|\,|m|^{\rho}\,|k|^\rho+|m|\,|k|^{\rho}\,|\ell|^\rho\,,\\ [5pt]
\forall (k,\ell,m)\,;k+\ell+m=0\,,
\end{array}
\end{equation}
\begin{equation}
\label{eq:est-1/2}
\begin{array}[t]{r}\dfrac{|k|k|^{2\rho}+\ell|\ell|^{2\rho}+m|m|^{2\rho}|}{(\min (|k|,|\ell|,|m|))^{1/2}}\lesssim |k|^{1/2}\,|\ell|^{\rho}\,|m|^\rho+|\ell|^{1/2}\,|m|^{\rho}\,|k|^\rho+|m|^{1/2}\,|k|^{\rho}\,|\ell|^\rho\,,\\ [5pt]
\forall (k,\ell,m)\,;k+\ell+m=0\,,\;k\ell m\neq 0\,,
\end{array}
\end{equation}
which both follow from \cite[Proposition 3]{Hunter2006}.
We thus see from the symmetry of $p=p_1+p_2$, the estimates of $p_1$ and $p_2$ in \Pun\, \Pde, and the general estimates in \eqref{eq:est0} and \eqref{eq:est-1/2} that
$$\frac{\dif }{\dif s} \|u\|_{\dot{H}^\rho}^2  \begin{array}[t]{l} =
2 \Real \left( i \iint p(-k,k-\ell,\ell)\,k|k|^{2\rho} \hu(-k,s) \,\hu(k-\ell,s)\,\hu(\ell,s)\,\dif \ell\,\dif k\right)  \\ [5pt]
=
-\frac23 \Real \begin{array}[t]{l}\left( i \iint p(-k,k-\ell,\ell)\,(-k|k|^{2\rho}+(k-\ell)|k-\ell|^{2\rho}+\ell|\ell|^{2\rho})\right. \\ [5pt]
\left.\,\hu(-k,s) \,\hu(k-\ell,s)\,\hu(\ell,s)\,\dif \ell\,\dif k\right)
\end{array}\\ [5pt]
\lesssim
\iint |k|^{1/2}(1+|k|) |\hu(-k,s)| \,|k-\ell|^{\rho} |\hu(k-\ell,s)|\,|\ell|^\rho \hu(\ell,s)\,\dif \ell\,\dif k\,,
\end{array}$$
where we have omitted to write the other two integrals since they are all equal to each other by the changes of variables $(-k,\ell)\mapsto (k-\ell,\ell)$ and  $(-k,\ell)\mapsto (-k,k-\ell)$.
Hence 
$$\frac{\dif }{\dif s} \|u\|_{\dot{H}^\rho}^2 \lesssim 
\|u\|_{\dot{H}^\rho}^2\;\int |k|^{1/2}(1+|k|) |\hu(-k,s)| \dif k\,$$
by the Fubini theorem and Cauchy--Schwarz inequality. This eventually gives \eqref{eq:estrho} since by the Cauchy-Schwarz inequality again
$$\int |k|^{1/2}(1+|k|) |\hu(-k,s)| \dif k\,\lesssim\,\|u\|_{H^\sigma}$$
for $\sigma>2$, this lower bound ensuring that  $\int |k|(1+|k|)^{2(1-\sigma)} \dif k<+\infty$.

Once we have these estimates, and similar ones concerning the `linearized' equation
$$\hu_s(k,s)\,-\,i\,k\,\int_{\R}\,p(-k,k-m,m)\, \hv(k-m,s)\,\hu(m,s)\,\dif m\,=\,0\,,$$
we can follow the same regularizing method as in \cite{Benzoni2009}, which then plays the role of the Galerkin method used in \cite{Hunter2006}.
\end{proof}

\begin{corollary}
In the framework of Theorem \ref{mainthm}, the amplitude equation \eqref{eq:amplitude} is locally well-posed in the  inverse image of $H^{\sigma}(\R))$ by $|\partial_y|^{1/2}$, for $\sigma>2$.
\end{corollary}

\begin{remark} {\rm In the case when $b_2=0$, the regularity index can be lowered by one. Indeed, the a priori estimates for  \eqref{eq:amplitudeb}  then reduce to 
$$\frac{\dif }{\dif s} \|u\|_{\dot{H}^\rho}^2  
\lesssim
\iint |k|^{1/2} |\hu(-k,s)| \,|k-\ell|^{\rho} |\hu(k-\ell,s)|\,|\ell|^\rho \hu(\ell,s)\,\dif \ell\,\dif k\,,
$$
so that we only need that 
$\int |k|^{1/2} |\hu(-k,s)| \dif k$
be finite for $u\in H^{\sigma}(\R))$, which is ensured as soon as $\sigma>1$.
}
\end{remark}

\begin{remark}\label{rem:Burgers} {\rm In the case when $b_1=0$,
there is another way to go from the amplitude equation \eqref{eq:amplitude} to an equation looking like \eqref{eq:amplitudeb}, which breaks the symmetry of the kernel but still yields a  well-posedness result. Let us recall indeed from the proof of Lemma \ref{lem:pp}  that
$$|b_2(\xi_1,\xi_2,\xi_3)| \lesssim  \frac{|\xi_1\xi_2\xi_3|}{|\xi_1| +|\xi_2|+|\xi_3|}\,,$$
hence
$$|b_2(\xi_1,\xi_2,\xi_3)| \lesssim {|\xi_2\xi_3|}\,,$$
which enables us to define a bounded kernel of degree zero $q$ by
$$q(k-\ell,\ell):= \frac{b(-k,k-\ell,\ell)}{|k-\ell||\ell|}\,.$$
In addition, this kernel satisfies the crucial estimate
\begin{equation}\label{eq:crucial}
|q(k,\ell)-q(-\ell-k,\ell)|\lesssim |\ell/k|\,, \;0<|\ell|<|k|\,.
\end{equation}
This comes from the fact that, by the symmetry of $b$,
$$|k| q(k,\ell) = \frac{b(-k-\ell,k,\ell)}{|\ell|}=\frac{b(k,-k-\ell,\ell)}{|\ell|} =|k+\ell| q(-\ell-k,\ell)\,,$$
hence
$$|q(k,\ell)-q(-\ell-k,\ell)|=\frac{|q(-\ell-k,\ell)|}{|k|}\,||k+\ell|-|k||\leq \|q\|_{L^\infty}\,\frac{|\ell|}{|k|}\leq\,\|q\|_{L^\infty}\,, \;0<|\ell|<|k|\,.$$
Now, for $w$ to solve \eqref{eq:amplitude}, $v:=|\partial_y|w$ must solve the \emph{nonlocal Burgers equation}
\begin{equation}\label{eq:amplitudeB}
v_s = ( {\mathcal Q}[v])_y\,,
\end{equation}
where ${\mathcal Q}$ is defined by
$$\widehat{{\mathcal Q}[v]}(k)=\int q(k-m,m)\,\hv(k-m)\,\hv(m)\,\dif m\,.$$
By substituting \eqref{eq:crucial} for the piecewise ${\mathcal C}^1$ assumption on $q$ in the main result in \cite{Benzoni2009}, we can still prove by the same method that \eqref{eq:amplitudeB} is locally well-posed in $H^2(\R)$. See \cite{Marcou} for a similar result in a periodic setting. This result might even be extended to fractional regularity indices $s>3/2$, as for the standard Burgers equation. This would eventually yield local well-posedness for \eqref{eq:amplitude} in the inverse image of $H^{s}(\R))$ by $|\partial_y|$. This is a slightly smaller subspace of $H^{s+1}(\R)$ than the  inverse image of $H^{s+1/2}(\R)$ by $|\partial_y|^{1/2}$, the difference coming only from the low frequency behavior of their elements.
}
\end{remark}

\begin{remark}{\rm 
A piecewise continuous kernel satisfying \eqref{eq:crucial} automatically satisfies Hunter's stability condition
$$\Hunter \qquad  q(1,0+)=q(-1,0+)\,,$$
which was coined in \cite{Hunter}. `Conversely', as was pointed out in \cite{Benzoni2009}, piecewise ${\mathcal C}^1$  kernels satisfying Hunter's stability condition automatically satisfy \eqref{eq:crucial}.}
\end{remark}

\begin{remark}{\rm 
Even though the Oseen--Frank energy considered in \cite{Saxton,AliHunter-crystals} satisfies our general assumptions, the associated model, the so-called director-field system, does not readily falls within our framework because of the constraint $|\bn|=1$. However, the kernel obtained in \cite{Austria} has the expected properties, namely, homogeneity degree one, and \Pun\ for the associated, rescaled kernel.
}
\end{remark}

\section{Surface waves at reversible phase boundaries}\label{s:pt}

\subsection{Phase boundaries \emph{versus} classical shocks}
This part is devoted to a fluid model that can also be derived from a variational principle, but not as a 
simple one as in the previous section, unfortunately. That model describes the dynamics of reversible 
isothermal phase boundaries in compressible fluids. Mathematically, it amounts to a quasilinear, free 
boundary hyperbolic problem in which phase boundaries can be viewed as \emph{undercompressive 
shocks} - which means that the number of outgoing characteristics is not lower than the number of 
incoming ones, contrary to what happens for classical shocks. The linearized problems about planar 
phase boundaries were investigated in \cite{Benzoni1998}, where linear surface waves were found by 
explicit computations. It is notable that no such waves exist for classical - or Lax -  shocks in compressible 
fluids. More precisely, there are no \emph{neutral modes} associated with Lax shocks in ideal gases at all, 
and there can only be neutral modes of \emph{infinite energy} in gases obeying more general pressure 
laws (see \cite[\S~15.2]{BS}). By neutral modes we mean here solutions of linearized problems about 
planar shocks that oscillate as $\ee^{i(\tau t + \bfeta\cdot \bx)}$ in the direction of shock fronts. Those 
of infinite energy also oscillate in the transverse direction to shock fronts, unlike genuine surface waves.

The difference between Lax shocks and phase boundaries can be explained from various perspectives. 
As far as neutral modes are concerned, we can invoke the fact that the energy - in fact the sum of the 
kinetic energy density and the free energy density - is conserved across reversible isothermal phase 
boundaries. This implies that the associated free boundary value problems can be derived from a variational 
principle for the Lagrangian
$$\Lag[\rho,\vit]:=\int_{0}^{T}\int_{\Omega} \left(\frac12 \rho |\vit|^2-F(\rho,\theta)\right)\,\dif \bx\,\dif t\,,$$
where $\rho$ denotes the density, $\vit$ the velocity, and $F(\rho,\theta)$ the free energy density at 
temperature $\theta$. More precisely, we can derive the conservation of the momentum $\rho\vit$ and of 
the total energy $\frac12 \rho |\vit|^2+F(\rho,\theta)$ by assuming that $\Lag$ has a critical point at a state 
$(\rho,\vit,\theta)$ such that $\theta$ is transported by $\vit$ and the pair $(\rho,\vit)$ satisfies the continuity 
equation 
$$\partial_t\rho+\nabla\cdot(\rho\vit)=0\,.$$
Indeed, it suffices to substitute free energy and temperature for internal energy and entropy in \cite{Serre-M2AN}. 
Even though there are whole books on variational principles (for instance the valuable monographs 
\cite{BerdichevskyI,BerdichevskyII}), up to the authors' knowledge and in their opinion, Serre's approach 
yields the most rigorous variational derivation of conservation laws in fluids. It is even valid for weak solutions, 
and thus shows local conservation laws as well as jump conditions. As pointed out in \cite[p.746]{Serre-M2AN}, 
the equations derived this way do not support classical shocks in barotropic fluids - in particular isothermal 
ones -, since the energy is not conserved across those shocks. In other words, we cannot say that isothermal 
Lax shocks are governed by a variational principle. By contrast, reversible isothermal phase boundaries are, 
and this variational nature supports the existence of surface waves, even though it does not obviously follow 
from a general theory as in \cite{Serre-JFA} - reported as Theorem \ref{thm:Serre} in the present paper. 
(One may observe that Lax shocks in non-barotropic fluids are governed by a variational principle, as shown 
in \cite{Serre-M2AN}, and are not associated with surface waves.)

\subsection{Glimpse of the amplitude equation}

Weakly nonlinear surface waves for free boundary value problems were addressed in \cite{BenzoniRosini}, 
in the same spirit as in the seminal work by Hunter \cite{Hunter} for hyperbolic boundary value problems with 
fixed boundaries. In particular, an amplitude equation was derived for weakly nonlinear surface waves 
associated with reversible, isothermal phase boundaries. Our aim here is to prove that this amplitude equation 
is actually well-posed.

In fact, as explained in \S~\ref{ss:derpt} below, the amplitude equation for this problem is `readily' found - up to 
nevertheless lengthy calculations - in nonlocal Burgers form \eqref{eq:amplitudeB}. Furthermore, its kernel $q$ 
is piecewise continuous and satisfies the estimate in \eqref{eq:crucial}. The latter is indeed a consequence - as 
in Remark \ref{rem:Burgers} above - of the boundedness of $q$ and its additional `symmetry'
\begin{equation}\label{eq:crucialsym}
|k| q(k,\ell) = |k+\ell| q(-\ell-k,-\ell)\,,\;\forall k,\ell\,;\;k\ell(k+\ell)\neq 0\,.
\end{equation}
More precisely, this kernel turns out to be of the form
\begin{equation}\label{eq:form}
q(k,\ell)=\left\{\begin{array}{l} \gamma\,,\;k>0\,,\;\ell>0\,,\\
\overline{\gamma}\,(1+\ell/k)\,,\;k>0\,,\;\ell<0\,,\;k+\ell>0\,,
\end{array}\right. 
\end{equation}
for some complex number $\gamma$, 
the values of $q$ in the other parts of $${\mathbb P}:=\R^2\backslash\{(k,\ell)\,;\;k\ell(k+\ell)\neq 0\}$$ being 
determined by the properties
\begin{equation}\label{eq:prop}
q(k,\ell)=q(\ell,k)\,,\;q(-k,-\ell)=\overline{q(k,\ell)}\,,\quad\forall k,\ell\,;\;k\ell(k+\ell)\neq 0\,.
\end{equation}
By direct inspection of the values of $q$ in the six connected parts of ${\mathbb P}$, which read
\begin{equation}\label{eq:qintegral}
q(k,\ell)=\left\{\begin{array}{l} \gamma\,,\;k>0\,,\;\ell>0\,,\\
\overline{\gamma}\,,\;k<0\,,\;\ell<0\,,\\
\overline{\gamma}\,(1+\ell/k)\,,\;k>0\,,\;\ell<0\,,\;k+\ell>0\,,\\
{\gamma}\,(1+\ell/k)\,,\;k<0\,,\;\ell>0\,,\;k+\ell<0\,,\\
\overline{\gamma}\,(1+k/\ell)\,,\;k<0\,,\;\ell>0\,,\;k+\ell>0\,,\\
{\gamma}\,(1+k/\ell)\,,\;k>0\,,\;\ell<0\,,\;k+\ell<0\,,\end{array}\right.
\end{equation}
we arrive indeed at the elementary result.

\begin{lemma}\label{lem:elem}
A function $q:\R^2\backslash\{(k,\ell)\,;\;k\ell(k+\ell)\neq 0\} \to \C$ satisfying \eqref{eq:form} and \eqref{eq:prop} 
is piecewise continuous, satisfies \eqref{eq:crucialsym}, and thus is bounded and satisfies \eqref{eq:crucial}.
\end{lemma}

Therefore, by repeating the proof of \cite[Theorems 3.1 \& 3.2]{Benzoni2009} under the sole assumption 
that $q$ is piecewise continuous and satisfies \eqref{eq:crucial} - instead of the piecewise ${\mathcal C}^1$ 
assumption together with Hunter's stability condition \Hunter\ - we find that the amplitude equation
$v_s = ( {\mathcal Q}[v])_y$ with 
$$
\widehat{{\mathcal Q}[v]}(k)=\int q(k-m,m)\,\hv(k-m)\,\hv(m)\,\dif m\,,
$$
and $q$ defined by \eqref{eq:form} and \eqref{eq:prop} is locally well-posed in $H^2(\R)$.

\subsection{Linear surface waves}\label{ss:sw}

Before turning to more details on weakly nonlinear analysis, let us recall the fully nonlinear model under 
consideration, and explain where the linear surface waves come from. This model is based on the isothermal 
Euler equations
\begin{equation}
\label{eq:euler}
\begin{cases}
\partial_t \rho +\nabla \cdot (\rho \, \vit) =0 \, ,& \\
\partial_t (\rho \, \vit) +\nabla \cdot (\rho \, \vit \otimes \vit) +\nabla p(\rho) =0 \, ,& 
\end{cases}
\end{equation}
where $p:=\rho \partial_\rho F -F$, of which smooth solutions are known to satisfy the local conservation law 
for the energy
\begin{equation}
\label{eq:energy}
\partial_t \big(\tfrac12 \rho \, |\vit|^2+F(\rho)\big) 
+\nabla \cdot \Big(\big(\tfrac12 \rho \, |\vit|^2+F(\rho)+p(\rho)\big) \vit \Big)=0\,.
\end{equation}
We have omitted to write the dependency of $F$ and $p$ on the temperature $\theta$, because $\theta$ is fixed. 
Phase boundaries arise when $p$ is not a monotone function of $\rho$ - for instance when $p$ obeys the van 
der Waals law under the critical temperature, but this specific example does not play any role in what follows. 
Mathematically speaking, isothermal reversible phase boundaries correspond to piecewise smooth weak 
solutions to \eqref{eq:euler}-\eqref{eq:energy}.

The starting point of the analysis is a pair of reference states $(\rho_\ell,\vit_\ell)$ and $(\rho_r,\vit_r)$  that 
correspond to a planar propagating discontinuity solving both \eqref{eq:euler} and \eqref{eq:energy} in a weak 
sense. For such a discontinuity to propagate at speed $\sigma$ in a direction $\bnu\in \R^d$, the states must 
satisfy the jump conditions
\begin{equation}
\label{eq:eulerjump}
[\rho (\vit\cdot \bnu-\sigma)]=0\,,\;[\rho (\vit\cdot \bnu-\sigma)\,\vit\,+\,p(\rho)\,\bnu]=0\,,
\end{equation}
\begin{equation}
\label{eq:energyjump}
[ (\vit\cdot \bnu-\sigma)\,(\tfrac12 \rho \, |\vit|^2+F(\rho)\big)\,+\,p(\rho)\,\vit\cdot\bnu]=0\,,
\end{equation}
where bracket expressions $[q]$ stand as usual for $q_r-q_\ell$. The two jump conditions in \eqref{eq:eulerjump} 
are the usual Rankine--Hugoniot conditions for \eqref{eq:euler}. The third jump condition \eqref{eq:energyjump} 
is \emph{not} satisfied by classical shock wave solutions to \eqref{eq:euler}. Given \eqref{eq:eulerjump}, 
\eqref{eq:energyjump} amounts to an \emph{equal area rule} in the thermodynamic variables $(1/\rho,p)$, which 
cannot be satisfied when $p(\rho)$ is increasing with $\rho$ and the mass flux
$$j:=\rho (\vit\cdot \bnu-\sigma)$$
across the discontinuity is nonzero. On the contrary, when $p(\rho)$ is decreasing on some interval and increasing 
outside this interval, we can find states $(\rho_\ell,\vit_\ell)$ and $(\rho_r,\vit_r)$ satisfying the three jump conditions 
- one may think of the smallest of the densities $\rho_\ell$ and $\rho_r$ as that of the gas phase, and the largest 
one as that of the liquid phase. In addition, we can find such states that are both \emph{subsonic} with respect to 
the front of discontinuity, which means that 
$$|\vit_{\ell,r}\cdot \bnu-\sigma|<c_{\ell,r}:=\sqrt{p'(\rho_{\ell,r})}\,.$$
The subscripts $\ell$ and $r$ here above may be thought of as abbreviations for `left' and `right', even though this 
is not meaningful in several space dimensions. For classical shocks we usually prefer the terms `behind' and `ahead', 
the state behind being subsonic and the state ahead of the shock being supersonic. For phase boundaries, both 
states being subsonic it is more natural to think of the `liquid' state and the `vapor' state to distinguish between them. 
However, the phase boundaries we consider are completely \emph{reversible}, which means that the states 
$(\rho_\ell,\vit_\ell)$ and $(\rho_r,\vit_r)$ can exchanged with each other.

From now on, we fix states like this, with $j\neq 0$. These are a basic example of what we call isothermal reversible 
\emph{dynamical} phase boundaries. The term `dynamical' refers to the fact that there is a nonzero mass flux $j$ 
across the boundary. Dynamical phase boundaries share the property $j\neq 0$ with classical shocks, whereas 
\emph{static} phase boundaries (with $j=0$) would be \emph{contact discontinuities}.

For convenience we denote by $$u_{\ell,r}:=\vit_{\ell,r}\cdot \bnu-\sigma$$ the relative velocities of the fluid with 
respect to the front of discontinuity, which we assume to be both positive without loss of generality - observe that 
$\rho_\ell u_{\ell}=\rho_r u_{r}=j\neq 0$ by assumption, and that the equations of motion 
\eqref{eq:euler}-\eqref{eq:energy}-\eqref{eq:eulerjump}-\eqref{eq:energyjump}  are invariant under the orthogonal 
symmetry defined by $\bnu$.

In the terminology of conservation laws, isothermal reversible dynamical phase boundaries are 
\emph{noncharacteristic}, and called \emph{undercompressive} because of the inequalities
$$u_{\ell,r} - c_{\ell,r} < 0 < u_{\ell,r} < u_{\ell,r} + c_{\ell,r}\,,$$
which mean that the number of characteristics of the isothermal Euler equations \eqref{eq:euler} is preserved 
across the front. This feature requires an `additional' jump condition apart from the Rankine--Hugoniot conditions. 
The conservation of energy in  \eqref{eq:energyjump}  provides such an additional jump condition.

The main question here is whether a planar propagating phase boundary persists under small perturbations of 
the front location and of the states on either side. This is a free boundary problem, which was first addressed in 
\cite{Benzoni1998}. (The same question for classical shocks was pointed out and solved by Majda in the 1980s.) 
The unknowns are the density and the velocity of the fluid on either side of the unknown boundary, supposedly 
close to $(\rho_\ell,\vit_\ell)$ and $(\rho_r,\vit_r)$ respectively, together with a scalar function $\Phi(\bx,t)$ close 
to $(\bx\cdot \bnu - \sigma t)$ such that $\Phi(\bx,t)=0$ is an equation for the unknown boundary. By a change 
of space-time variables depending on $\Phi$, the free boundary problem consisting of \eqref{eq:euler} - and 
thus automatically \eqref{eq:energy} - on either side of the boundary together with 
\eqref{eq:eulerjump}-\eqref{eq:energyjump} across the boundary can be changed into a boundary value problem 
with a fixed planar boundary of equation $z=0$. This boundary value problem can then be linearized about the 
reference solution corresponding to the discontinuous solution connecting $(\rho_\ell,\vit_\ell)$ to $(\rho_r,\vit_r)$. 
It is for this problem, referred to as (LBVP) hereafter, that surface waves were found in \cite{Benzoni1998}.

Without recalling all the notations from the earlier papers \cite{Benzoni1998,BenzoniRosini}, we can give an 
explicit form of surface waves. For $\bfeta\neq 0$ in the cotangent space to the fixed boundary, there exists 
$\tau\in \R\backslash \{0\}$ and a nontrivial solution to (LBVP) of the form $\ee^{i(\tau t+\bfeta\cdot \bx)} 
(R(z),\bU(z))$ with $(R(z),\bU(z))$ going to zero when $|z|$ goes to infinity, and more precisely
$$
(R(z),\bU(z))=\left\{\begin{array}{ll} 
\gamma_1\,\ee^{-\beta_1 z}\;(-i\, \tau +u_\ell \, \beta_1,i \, c_\ell^2 \, \bfeta - a_\ell \,\bnu)\;\,,& z<0\,,\\
\gamma_2\,\ee^{\beta_2 z}\;(-i\, \tau -u_r\, \beta_2,i \, c_r^2 \, \bfeta - a_r \,\bnu)\;\,,& z>0\,,
\end{array}\right.
$$
where
\begin{align}
& \beta_1 :=\dfrac{a_\ell -i\, u_\ell \, \tau}{c_\ell^2 -u_\ell^2} \, ,&
a_\ell :=-c_\ell \, \sqrt{(c_\ell^2 -u_\ell^2) \, |\bfeta|^2 -\tau^2} \, ,\notag\\
& \beta_2:=\dfrac{-a_r +i\, u_r \, \tau}{c_r^2 -u_r^2} \, ,&
a_r :=c_r \, \sqrt{(c_r^2 -u_r^2) \, |\bfeta|^2 -\tau^2} \, .\notag
\end{align}
With these notations the dispersion relation satisfied by $(\tau,\bfeta)$ reads
\begin{equation}
\label{LopatinskiiEuler}
u_\ell \, u_r \, a_\ell \, a_r +c_\ell^2 \, c_r^2 \, \tau^2 =0 \,,
\end{equation}
where the left-hand side is proportional to the Lopatinskii determinant  $\Delta(\tau,\bfeta)$ associated with (LBVP).

The coefficients $(\gamma_1,\gamma_2)$ are of course linked to each other through the linearized jump conditions.
Here we adopt a slightly different normalization compared with \cite[Section 3, Eq.~(3.51)]{BenzoniRosini} and rather 
choose
\begin{equation*}
\gamma_1 :=\dfrac{(\rho_r-\rho_l)\, u_r \, \tau}{u_r \, a_\ell -i\, c_\ell^2 \, \tau} \, ,\quad 
\gamma_2 :=\dfrac{-(\rho_r-\rho_l) \, u_\ell \, \tau}{u_\ell \, a_r -i\, c_r^2 \, \tau} \, .
\end{equation*}
This choice ensures that the scalar amplitude $w$ by means of which we can entirely determine the principal term 
in the weakly nonlinear asymptotic expansion of the fully nonlinear free boundary problem merely coincides with the 
partial derivative $\partial_y \chi$, where ${\chi}$ denotes the (first order) perturbation of $\Phi$ - recall that $\Phi=0$ 
is the free boundary equation - and $y$ is a placeholder for the phase $\tau \, t +\bfeta \cdot \bx$.

\subsection{Well-posedness of the amplitude equation}\label{ss:derpt}

Following \cite[Proposition 2.2]{BenzoniRosini}, we find that the evolution of the amplitude $w$ is then governed by 
a nonlocal Burgers equation
\begin{equation}
\label{BurgersEuler}
a_0(k) \, \widehat{w}_s (k,s) +\int_\R a_1(k-m,m) \, \widehat{w} (k-m,s) \, \widehat{w} (m,s) \, {\rm d}m =0 \, ,
\end{equation}
where $a_0$ and $a_1$ are given by Equations (2.24) and (2.25) of \cite[page 1471]{BenzoniRosini}. We are mainly 
concerned here with the structural properties of the amplitude equation \eqref{BurgersEuler}. Detailed computations 
leading to the final form of $a_0$ and $a_1$ given below can be found in a companion paper that is available online \cite{BenzoniCoulombel-note}.

The function $a_0$ in \eqref{BurgersEuler} is found to be of the form
\begin{equation*}
\forall \, k \neq 0 \, ,\quad a_0(k) =\dfrac{\Theta \, \alpha_0}{i\, k} \, ,\quad 
\alpha_0 := -\dfrac{1}{\tau} \, \left\{ u_\ell^2 \, u_r^2 \, \left( 
\dfrac{a_\ell^2}{c_\ell^2} +\dfrac{a_r^2}{c_r^2} \right) +2\, c_\ell^2 \, c_r^2 \, \tau^2 \right\} \, ,
\end{equation*}
where $\Theta$ is a nonzero real number, and $\alpha_0$ coincides with the $\tau$-derivative of the left-hand 
side in \eqref{LopatinskiiEuler}. This means that $a_0$ is proportional to $\partial_\tau\Delta(\tau,\bfeta)$, the 
$\tau$-derivative of the Lopatinskii determinant. This is consistent with our findings in \cite{BenzoniCoulombel}, 
even though the present framework does not fit those considered there, in particular because the pair $(\tau,\bfeta)$ 
is not in the elliptic region - transversally oscillating modes exist for $(\tau,\bfeta)$ but they are not present in the 
surface waves. The fact that $\Theta \alpha_0$ is obviously nonzero - as the product of a nonzero real number 
with the sum of positive real numbers - is reminiscent of the observation made for `standard' variational problems 
in \S~\ref{ss:deramp}, for which the amplitude equation is automatically evolutionary ($a(k)\neq 0$ for $k\neq 0$).

The expression of $a_1$ is given by:
\begin{equation*}
\dfrac{4\, \pi}{\Theta} \, a_1(k,m) =\begin{cases}
\alpha_1 & \text{if $k,m>0$,} \\
\overline{\alpha_1} \, (1+m/k) & \text{if $k>0$, $m<0$, $k+m>0$,}
\end{cases}
\end{equation*}
with
\begin{align*}
\alpha_1 := &
\dfrac{2}{(u_r-u_l)} \, (\tau^2 +u_\ell \, u_r \, |\bfeta|^2) \, i \, c_\ell^2 \, c_r^2 \, \tau\, 
\left( \dfrac{c_r^2 \, \gamma_2}{\rho_r \, u_r} -\dfrac{c_\ell^2 \, \gamma_1}{\rho_\ell \, u_\ell} \right)\\
&+ \left( \dfrac{p''(\rho_\ell)}{2} +\dfrac{c_\ell^2}{\rho_\ell} \right) \, u_\ell \, u_r \, \dfrac{a_r}{a_\ell} \, 
(\tau^2 +u_\ell^2 \, |\bfeta|^2) \, \gamma_1 \, (i\, \tau -u_\ell \, \beta_1) \\
& +\left( \dfrac{p''(\rho_r)}{2} +\dfrac{c_r^2}{\rho_r} \right) \, u_\ell \, u_r \, \dfrac{a_\ell}{a_r} \, 
(\tau^2 +u_r^2 \, |\bfeta|^2) \, \gamma_2 \, (i\, \tau +u_r \, \beta_2)\,.
\end{align*}
The above expressions for $a_0$ and $a_1$ show that \eqref{BurgersEuler} can be rewritten under the form
\begin{equation}\label{eq:defBurgers}
\widehat{w}_s (k,s) +i \, k \, \int_\R q(k-m,m) \, \widehat{w} (k-m,s) \, \widehat{w} (m,s) \, {\rm d}m =0 \, ,
\end{equation}
where $q$ satisfies \eqref{eq:form} with $\gamma:=\alpha_1/(4\, \pi \, \alpha_0)$.

Our definitive results can be thus summarized as follows.

\begin{theorem} For any planar discontinuity between states $(\rho_\ell,\vit_\ell)$ and $(\rho_r,\vit_r)$ that 
propagates at speed $\sigma$ in a direction $\bnu\in \R^d$ and solves \eqref{eq:eulerjump}-\eqref{eq:energyjump} 
with 
$$0<|\vit_{\ell,r}\cdot \bnu-\sigma|<\sqrt{p'(\rho_{\ell,r})}\,,$$
for all $(\tau,\bfeta)$ satisfying \eqref{LopatinskiiEuler}, there is a one-dimensional space of linear surface waves 
associated with the time frequency $\tau$ and the wave vector $\bfeta$ that solve a linearized version of 
\eqref{eq:euler}-\eqref{eq:eulerjump}-\eqref{eq:energyjump} about that planar discontinuity. The associated weakly 
nonlinear surface waves are governed by a nonlocal Burgers equation \eqref{eq:defBurgers} in which the kernel $q$ 
is of the form given in \eqref{eq:qintegral}. In particular, $q$ is bounded, satisfies the estimate in \eqref{eq:crucial}, 
and the nonlocal Burgers equation \eqref{eq:defBurgers} is locally well-posed in $H^2(\R)$.
\end{theorem}

\section*{Appendix}

\setcounter{proposition}{0}
\renewcommand{\theproposition}{A.\arabic{proposition}}
\setcounter{equation}{0}
\renewcommand{\theequation}{A.\arabic{equation}}

\begin{proposition}\label{prop:varderham}
Let $b:\R^3\to \C$ be a symmetric, continuous function outside $Z:=\{(k,\ell,m)\,;\;k\ell m = 0\}$ and such that, 
for all $(k,\ell,m)\in \R^3\backslash Z$,
$$b(-k,-\ell,-m)=\overline{b(k,\ell,m)}\,,$$ 
$$|b(k,\ell,m)|\,\leq\,C(1+k^2+\ell^2+m^2)\,,$$ 
with $C$ a positive constant. Then the functional
$${\mathcal T}[w]=\frac{1}{3}\,\iint b(-k-m,k,m)\,\hw(-k-m)\,\hw(k)\,\hw(m)\,\dif k\,\dif m\,$$
is well-defined on $H^1(\R)$, and its variational derivative $\delta {\mathcal T}$  is given by
$$\widehat{\delta {\mathcal T}[w]}(k)=2\pi\,\int b(-k,k-m,m)\,\hw(k-m)\,\hw(m)\,\dif m\,,\;\forall w\in H^2(\R;\R)\,.$$
\end{proposition}

\begin{proof} By the Cauchy--Schwarz inequality, $\hw\in L^1(\R)$ for all $w\in H^1(\R)$, and
$$\|\hw\|_{L^1}\,\leq\,\sqrt{\pi}\,\|w\|_{H^1}\,.$$
Therefore, using that 
$$|b(-k,k-m,m)|\,\leq\,C(1+|k|^2+|k-m|^2+|m|^2)\,\leq\,C(1+2|k||k-m|+2|k-m||m|+2|m||k|)\,,$$ 
we see that the trilinear mapping ${\mathcal T}$ is continuous on $H^1(\R)$, with
$$|{\mathcal T}[w]|\leq 8C\,\pi^2\,\|\hw\|_{L^1}\,\|w\|_{H^1}^2\,\leq 8C\,\pi^3\,\|w\|_{H^1}^3\,,$$
by the Fubini, Cauchy--Schwarz, and Plancherel  theorems.
Furthermore, for all $w$, $v\in H^1(\R;\R)$, we have
$$\frac{\dif }{\dif \theta} {\mathcal T}[w+\theta v]_{|\theta=0}\begin{array}[t]{l}\,=\,\begin{array}[t]{l}
\frac{1}{3}\,\iint b(-k-m,k,m)\,\hv(-k-m)\,\hw(k)\,\hw(m)\,\dif k\,\dif m\\[10pt]
+\frac{1}{3}\,\iint b(-k-m,k,m)\,\hw(-k-m)\,\hv(k)\,\hw(m)\,\dif k\,\dif m\\[10pt]
+\frac{1}{3}\,\iint b(-k-m,k,m)\,\hw(-k-m)\,\hw(k)\,\hv(m)\,\dif k\,\dif m\\[10pt]
\end{array}\\
\,=\,
\iint b(-k-m,k,m)\,\hw(-k-m)\,\hw(k)\,\hv(m)\,\dif k\,\dif m\end{array}$$
by the symmetry of $b$ and obvious changes of variables. The integral above is well defined for all $w$, 
$v\in H^1(\R;\R)$, as expected from the fact that ${\mathcal T}$ is differentiable since it is trilinear continuous. 
However, the definition of its variational derivative $\delta {\mathcal T}[w]$ is more demanding on $w$. It amounts 
to rewriting
$$\frac{\dif }{\dif \theta} {\mathcal T}[w+\theta v]_{|\theta=0}\,=\,\int \delta {\mathcal T}[w] \;v(y)\,\dif y\,,$$
so that $\delta {\mathcal T}[w]$ bears all the derivatives. In view of the large frequency behavior of the kernel $b$, 
it turns out that this is possible as soon as $w$ belongs to $H^2$, as we show below.

Let us recall that, by the Plancherel theorem,
$$\textstyle \int \widehat{f}(-m)\,\hv(m)\,\dif m\,=\,2\pi\,\int f(y)\,v(y)\,\dif y$$
for all real valued, square integrable functions $f$ and $v$.
We claim that for $w\in H^2(\R;\R)$, we can define a real valued $f\in L^2$ by 
$$\widehat{f}(m)\,=\,\int b(-m,m-k,k)\,\hw(m-k)\,\hw(k)\,\dif k\,.$$
Indeed, using that 
$$|b(-m,m-k,k)|\,\leq\,C(1+m^2+(m-k)^2+k^2)\,\leq\,C(1+3(m-k)^2+3k^2)\,,$$
we find that
$$\textstyle \left|\int b(-m,m-k,k)\,\hw(m-k)\,\hw(k)\,\dif k\right|\leq C \big( (|\hw|*|\hw|)(m) 
+ 6 (|\hw|*|\widehat{w_{yy}}|)(m)\big)\,,$$
and 
$$\||\hw|*|\hw|\|_{L^2}\leq \|\hw\|_{L^1} \|\hw\|_{L^2} = 2\pi \|\hw\|_{L^1} \|w\|_{L^2}\,,$$
$$\||\hw|*|\widehat{w_{yy}}|\|_{L^2}\leq \|\hw\|_{L^1} \|\widehat{w_{yy}}\|_{L^2} \leq 2\pi \|\hw\|_{L^1} \|w\|_{H^2}\,.$$
This shows that
$$\textstyle \int \left(\int b(m,-m-k,k)\,\hw(-m-k)\,\hw(k)\,\dif k\right)\,\hv(m)\,\dif m\,=\,2\pi\,\int f(y)\,v(y)\,\dif y$$
for all $w\in H^2(\R;\R)$.
By the Fubini theorem and the symmetry of $b$, the left-hand side is exactly what we have found for the 
directional derivative of ${\mathcal T}$. We thus have
$$\frac{\dif }{\dif \theta} {\mathcal T}[w+\theta v]_{|\theta=0}\,=\,2\pi\,\int f(y)\,v(y)\,\dif y\,,$$
which means that $\delta {\mathcal T}[w]=2\pi f$.
\end{proof}

\begin{proposition}\label{prop:vardermom}
Let us consider the functional ${\mathcal M}$ defined by 
$${\mathcal M}[w]=\frac{1}{2}\,\iint |k|\,|\hw(k)|^2\,\dif k\,$$
for all $w\in H^{1/2}(\R)$. Then its variational derivative $\delta {\mathcal M}$  is such that
$$\partial_y w = - \tfrac{1}{2\pi}\, {\mathcal H}({\delta {\mathcal M}[w]})\,,\;\forall w\in H^1(\R;\R)\,.$$
\end{proposition}

\begin{proof}
The computations are similar to, and simpler than in the previous proposition. We have
$$\frac{\dif }{\dif \theta} {\mathcal M}[w+\theta v]_{|\theta=0}\,=\,
\iint |k|\,\hw(-k)\,\hv(k)\,\dif k\,=\,2\pi\,\iint u(y)\,v(y)\,\dif y\,,
$$
provided that $\widehat{u}(-k)=|k|\hw(-k)$ \emph{a.e}, or equivalently,
$ik \hw(k)=i\sgn(k)\widehat{u}(k)$, that is,
$\partial_y w = - {\mathcal H}(u)=- \tfrac{1}{2\pi}\, {\mathcal H}({\delta {\mathcal M}[w]})$.
\end{proof}

\paragraph{Acknowledgement.} This work has been supported by the ANR project BoND (ANR-13-BS01-0009-01).

\bibliographystyle{plain}
\bibliography{BC2}
\end{document}